\documentclass[11pt]{article}
\usepackage{amsmath, amssymb,epsfig,bm}
\usepackage{latexsym}

\usepackage{graphicx}
\usepackage{color}
\usepackage{booktabs} 
\usepackage{algorithm}
\usepackage{algpseudocode}
\usepackage{subcaption}
\usepackage{graphicx} 
\usepackage{url}
\usepackage[T1]{fontenc}

\numberwithin{equation}{section}

\def\pproof{\hfill {\small $\Box$} \\}

\def\m{\mbox}

\newtheorem{lemma}{Lemma}[section]

\newtheorem{cor}{Corollary}[section]
\newtheorem{theorem}{Theorem}[section]
\newtheorem{rem}{Remark}[section]
\newtheorem{example}{Example}[section]

\newcommand{\lla}{\|{\hskip -1pt}|}
\newcommand{\rra}{\|{\hskip -1pt}|}

\newcommand{\vep}{\varepsilon}

\newcommand{\diam}{\mathrm{diam}}

\newcommand{\E}{{\mathbb{E}}}

\newcommand{\sd}{\mathsf{d}}

\newcommand{\be}{\begin{eqnarray}}
\newcommand{\ee}{\end{eqnarray}}
\newcommand{\beq}{\begin{equation}}
\newcommand{\eeq}{\end{equation}}
\newcommand{\ben}{\begin{eqnarray*}}
\newcommand{\een}{\end{eqnarray*}}

\newtheorem{assu}{Assumption}

\numberwithin{assu}{section}
\numberwithin{theorem}{section}
\numberwithin{alg}{section}

\begin{document}

\title{Stochastic Convergence Analysis for Large-Scale Linear Discrete Ill-posed Problems
}

\author{
Duan-Peng Ling\thanks{Department of Mathematics,
Southern University of Science and Technology (SUSTech),
1088 Xueyuan Boulevard, 
University Town of Shenzhen,
Xili, Nanshan, Shenzhen, Guangdong Province, P.R.China. (lingdp2023@mail.sustech.edu.cn).
}
\and Wenlong Zhang\thanks{Corresponding author. Department of Mathematics $\&$  National Center for Applied Mathematics Shenzhen, Southern University of Science and Technology (SUSTech), 1088 Xueyuan Boulevard, University Town of Shenzhen, Xili, Nanshan, Shenzhen, Guangdong Province, P.R.China. (zhangwl@sustech.edu.cn).  The work of Zhang was supported by the National Natural Science Foundation of China under grant numbers No.12371423, No.12241104 and No.12561160122 and the fund of the Guangdong Provincial Key Laboratory of Computational Science and Material Design (No.2019030301001).
}
}

\date{}
\maketitle

\begin{abstract}
We study weighted Tikhonov regularization for large-scale linear discrete ill-posed problems with random noise. Under a polynomial upper-bound assumption on the generalized eigenvalues of the discrete forward operator, we derive stochastic error bounds for two noise models: expectation bounds for independent zero-mean bounded-variance noise, and high-probability bounds for independent sub-Gaussian noise. The analysis yields an a priori parameter-choice rule and suggests an adaptive strategy suitable for large-scale computation. Numerical experiments support the theory and show that the predicted parameter is nearly optimal and that the adaptive method is effective in practice.  

\end{abstract}



\section{Introduction}

Inverse problems play an important role in many areas of science and engineering, such as astronomy, geoscience, biomedical imaging, mining engineering, and medical diagnostics; see, for example, \cite{andrews1977digital,boden1996massively,bose1998high,buzug2008computed,chung2006numerical,epstein2007,haber2014computational,hansen2006deblurring,hansen2018air,Jefferies2002blind,Wang2015regularization}. A fundamental objective in the study of inverse problems is to recover physically meaningful information from indirect measurements, which are typically contaminated by noise and other sources of uncertainty. In this paper, we consider a broad class of inverse problems that, after discretization, can be formulated as large-scale linear systems of the form
\begin{equation}\label{eq:data}
b=Ax^*+e,
\end{equation}
where \(x^* \in \mathbb{R}^n\) denotes the unknown parameter vector or quantity of interest, \(b \in \mathbb{R}^n\) represents the observed data, and \(A \in \mathbb{R}^{n\times n}\) is the discretized forward operator associated with the underlying physical or measurement process. The perturbation term \(e = \{e_i\}^n_{i=1} \in \mathbb{R}^n\) accounts for measurement noise, model errors, discretization effects, and finite-precision inaccuracies. In many applications, the matrix \(A\) is ill-conditioned and the data \(b\) are contaminated by noise, rendering direct inversion highly unstable. As a result, the discretized system is typically a linear discrete ill-posed problem, and this has motivated a substantial body of work on regularization methods and parameter-choice strategies. Classical contributions in this direction include generalized cross-validation and the L-curve criterion; see, for example, \cite{GolubHeathWahba1979,Hansen1992,HansenOLeary1993}.

Regularization is therefore indispensable for computing meaningful approximations. Among the available approaches, Tikhonov-type methods remain central because they offer a robust balance between data fidelity and stability and admit a transparent spectral interpretation. In this work we consider the weighted least-squares formulation
\begin{equation}\label{p1}
x_\lambda
=
\arg\min_{x\in\mathbb{R}^n}
\left\{
\|Ax-b\|^2+\lambda\|x\|_W^2
\right\},
\end{equation}
where \(\lambda>0\) is the regularization parameter, \(W\) is a symmetric positive-definite weight matrix, and \(\|x\|_W^2 := x^\top W x\). This formulation includes classical Tikhonov regularization as the special case \(W=I\), while also allowing weighted penalties that incorporate prior structural information on the unknown. The practical success of this method depends crucially on the choice of the parameter \(\lambda\), which must balance data fidelity against stability. If \(\lambda\) is chosen too small, the reconstruction is dominated by noise; if it is too large, important features of the exact solution are oversmoothed.

The computational difficulty becomes more pronounced in large-scale inverse problems, where fine discretizations or high-dimensional measurements make direct factorizations and full spectral decompositions impractical. In this regime, iterative regularization and hybrid projection methods have become fundamental tools. Representative contributions include hybrid LSMR for large-scale Tikhonov regularization, generalized hybrid methods for Bayesian inverse problems, recycling-based hybrid projection methods, and the recent survey on hybrid projection methods for large-scale inverse problems \cite{ChungPalmer2015,ChungSaibaba2017,JiangChungDeSturler2021,ChungGazzola2024}.

Recent work has broadened this picture in several directions. On the algorithmic side, projected and iterated Tikhonov schemes with adaptive parameter choice have been developed for large-scale problems, together with refined convergence analyses for Arnoldi--Tikhonov-type methods \cite{BucciniGazzolaOniskPashaReichel2025,BianchiDonatelliFurchiReichel2025}. On the stochastic side, early stopping rules for stochastic mirror descent methods have been analyzed for ill-posed inverse problems \cite{HuangJinLuZhang2025}. At the same time, data-driven inversion has renewed attention to regularization methods that remain convergent in the classical sense \cite{HauptmannMukherjeeSchoenliebSherry2025}. These developments highlight the need for analyses that explicitly connect stochastic noise models, spectral decay, and computationally feasible regularization strategies.

Despite recent progress, stochastic convergence results for large-scale linear discrete ill-posed problems that simultaneously capture the effects of the problem size, noise strength, regularization parameter, and spectral decay remain relatively limited. This issue is particularly relevant when the data perturbation is random and one seeks estimates that are informative both theoretically and computationally, and that can also guide parameter selection.

In this paper, we investigate a Tikhonov-type regularization method for large-scale linear inverse problems with random noise. Our analysis is carried out under the following polynomial upper-bound assumption on the generalized eigenvalues associated with the discrete forward operator.
\begin{assu}\label{Assumption1}
The eigenvalues \(\rho_1 \geq \rho_2 \geq \cdots \geq \rho_n\) of the generalized eigenvalue problem
\begin{equation}\label{ccc}
(A\psi,Au)=\rho(\psi,u)_W \qquad \forall u\in \mathbb{R}^n,
\end{equation}
satisfy
\[
\rho_k \leq C k^{-\alpha}, \qquad k=1,2,\dots,n,
\]
for some constant \(\alpha>1\).
\end{assu}
We consider two noise models:
\smallskip

({\bf M1}) \(\{e_i\}_{i=1}^n\) are independent random variables satisfying
\(\mathbb{E}[e_i]=0\) and \(\mathbb{E}[e_i^2]\leq \sigma^2\);
\smallskip

({\bf M2}) \(\{e_i\}_{i=1}^n\) are independent sub-Gaussian random variables with parameter \(\sigma\).

To facilitate a consistent comparison across different discretization levels and to reduce scaling effects induced by the discretization of the forward operator, we use the normalized Euclidean norm \(\frac{1}{\sqrt n}\|x\|\). Throughout the paper, \(C\) denotes a generic positive constant whose value may change from line to line.

The main contributions of this paper are as follows. First, under the bounded-variance noise model {\bf (M1)}, we establish expectation-based stochastic convergence estimates for Tikhonov-regularized solutions, with explicit dependence on the noise level, the regularization parameter, and the spectral decay of the forward operator. Second, under the sub-Gaussian noise model {\bf (M2)}, we derive stronger high-probability bounds by combining concentration arguments with covering entropy estimates. Third, these results yield an a priori parameter-choice rule and further motivate an adaptive parameter-selection algorithm that is computationally feasible for large-scale problems. Finally, we validate the theory by numerical experiments on representative linear inverse problems, demonstrating both the near-optimality of the predicted parameter choice and the practical effectiveness of the proposed adaptive method.

The remainder of this paper is organized as follows. In Section \ref{sec:bv}, we study the case of independent noise with bounded variance and derive expectation-based stochastic convergence estimates for the weighted Tikhonov regularized solutions. Section \ref{sec:gaussian} is devoted to the sub-Gaussian noise setting, where we combine concentration arguments with covering entropy estimates to establish sharper high-probability error bounds and the corresponding a priori parameter-choice rule. In Section \ref{sec:adaptive}, motivated by the theoretical analysis, we introduce a data-driven adaptive procedure for selecting the regularization parameter in large-scale problems. Finally, Section \ref{sec:num} presents numerical experiments, illustrating the near-optimality of the predicted parameter choice and the practical effectiveness of the proposed adaptive method.

\section{Stochastic convergence for noisy data of variables with bounded variance}\label{sec:bv}
In this section, we consider the bounded-variance noise model {\bf (M1)} and derive expectation-based stochastic convergence estimates for the error.

\begin{lemma}\label{lemma:eigv}
Let \( \rho_1 \ge \rho_2 \ge \cdots \ge \rho_n \) be the eigenvalues of the generalized eigenvalue problem
\[
 (A\psi,Au) = \rho(\psi,u)_W , \quad \forall u\in\mathbb{R}^n.
 \]
Then the corresponding eigenvectors \(\{\psi_k\}_{k=1}^n\) form an orthonormal basis of \(\mathbb{R}^n\) under the inner product \((\cdot,\cdot)_W\), and satisfy \((A\psi_i, A\psi_j) = \rho_i \delta_{ij}\).
\end{lemma}
\begin{proof}

The eigenvalue problem (\ref{ccc}) is equivalent to the generalized matrix eigenvalue problem
\[
A^\top A\psi = \rho W\psi.
\]

Define \(M := W^{-1/2} A^\top A W^{-1/2}.\) The matrix \(M\) is symmetric and positive semidefinite. Hence, there exist real eigenvalues \(\rho_1 \ge \rho_2 \ge \cdots \ge \rho_n \ge 0\) and an orthonormal basis \(\{z_k\}_{k=1}^n\) of \(\mathbb{R}^n\) such that
\[
Mz_k = \rho_k z_k, \qquad (z_i,z_j)=\delta_{ij}.
\]
Define \(\psi_k := W^{-1/2} z_k\), \(k=1,\dots,n.\) Then \(\psi_k\) satisfies
\[
A^\top A\psi_k = \rho_k W\psi_k,
\]
and hence is an eigenvector of the generalized eigenvalue problem (\ref{ccc}).
Let \(\psi_i\) and \(\psi_j\) be eigenvectors corresponding to eigenvalues \(\rho_i\) and \(\rho_j\), respectively. Then for any \((i,j)\),
\[
(\psi_i,\psi_j)_W
= (W^{1/2}\psi_i, W^{1/2}\psi_j)
= (z_i,z_j)
= \delta_{ij}.
\]
Therefore, the eigenvectors \(\{\psi_k\}_{k=1}^n\) form an orthonormal basis of \(\mathbb{R}^n\) under the inner product \((\cdot,\cdot)_W\). Moreover,
\[
(A\psi_i,A\psi_j)=\rho_i(\psi_i,\psi_j)_W = \rho_i\delta_{ij}.
\]
\pproof
\end{proof}

\begin{theorem}\label{thm:2.1}
Suppose that Assumption \ref{Assumption1} holds, and let $x_\lambda \in X$ be the unique minimizer of \eqref{p1}. Then there exist constants $\lambda_0>0$ and $C>0$, independent of $n$, $\lambda$, and $\sigma$, such that for every $0<\lambda\le \lambda_0$,
\begin{align}
\mathbb{E}\Big[\frac1n\|Ax_\lambda-Ax^*\|^2\Big]
&\le C\lambda \frac1n\|x^*\|_W^2
   + \frac{C\sigma^2}{n\lambda^{1/\alpha}},
\label{p5}\\
\mathbb{E}\Big[\frac1n\|x_\lambda-x^*\|_W^2\Big]
&\le C\frac1n\|x^*\|_W^2
   + \frac{C\sigma^2}{n\lambda^{1+1/\alpha}}.
\label{p6}
\end{align}
\end{theorem}

\begin{proof}
By the first-order optimality condition for the quadratic minimization problem \eqref{p1}, the unique minimizer $x_\lambda\in X$ satisfies
\begin{equation}\label{p7}
\lambda (x_{\lambda},u)_W+(Ax_{\lambda},Au) =(b,Au)
\qquad \forall u\in \mathbb{R}^n.
\end{equation}
For $u\in \mathbb{R}^n$, define the energy norm
\[
\lla u\rra_\lambda^2:=\lambda (u,u)_W+\|Au\|^2.
\]
Taking $u=x_\lambda-x^*$ in \eqref{p7}, and using \eqref{eq:data}, we obtain
\begin{equation}\label{x4}
\lla x_{\lambda}-x^*\rra_{\lambda}
\le \lambda^{1/2}\|x^*\|_W+\sup_{v\in \mathbb{R}^n}\frac{(e,Av)}{\lla v\rra_{\lambda}}.
\end{equation}

Let $\rho_1\ge \rho_2\ge \cdots \ge \rho_n$ be the eigenvalues of
\begin{equation}\label{x2}
\rho(\psi,u)_W=(A\psi,Au)
\qquad \forall u\in \mathbb{R}^n,
\end{equation}
and let $\{\psi_k\}_{k=1}^n$ be the corresponding eigenvectors. By Lemma \ref{lemma:eigv},
\[
(\psi_k,\psi_l)_W=\delta_{kl},
\qquad
(A\psi_k,A\psi_l)=\rho_k\delta_{kl}.
\]
Hence, for any $v\in \mathbb{R}^n$, we may write
\[
v=\sum_{k=1}^n v_k\psi_k,
\qquad
\lla v\rra_\lambda^2=\sum_{k=1}^n(\lambda+\rho_k)v_k^2.
\]
Using the Cauchy--Schwarz inequality, we get
\begin{align*}
(e,Av)^2
&=
\left(
\sum_{k=1}^n v_k \sum_{i=1}^n e_i (A\psi_k)_i
\right)^2 \\
&\le
\Big(\sum_{k=1}^n (\lambda+\rho_k)v_k^2\Big)
\Big(\sum_{k=1}^n (\lambda+\rho_k)^{-1}\Big(\sum_{i=1}^n e_i(A\psi_k)_i\Big)^2\Big) \\
&=
\lla v\rra_\lambda^2
\sum_{k=1}^n (\lambda+\rho_k)^{-1}\Big(\sum_{i=1}^n e_i(A\psi_k)_i\Big)^2.
\end{align*}
Therefore,
\[
\sup_{v\in\mathbb{R}^n}\frac{(e,Av)^2}{\lla v\rra_\lambda^2}
\le
\sum_{k=1}^n (\lambda+\rho_k)^{-1}\Big(\sum_{i=1}^n e_i(A\psi_k)_i\Big)^2.
\]
Taking expectations and using the independence, zero mean, and variance bound $\mathbb{E}[e_i^2]\le \sigma^2$, we obtain
\begin{align*}
\mathbb{E}\Big[\sup_{v\in\mathbb{R}^n}\frac{(e,Av)^2}{\lla v\rra_\lambda^2}\Big]
&\le
\sum_{k=1}^n (\lambda+\rho_k)^{-1}
\mathbb{E}\Big(\sum_{i=1}^n e_i(A\psi_k)_i\Big)^2 \\
&\le
\sigma^2\sum_{k=1}^n \frac{\rho_k}{\lambda+\rho_k},
\end{align*}
where we also used $\|A\psi_k\|^2=\rho_k$.

By Assumption \ref{Assumption1},
\[
\sum_{k=1}^n \frac{\rho_k}{\lambda+\rho_k}
\le
C\sum_{k=1}^n \frac{1}{1+\lambda k^\alpha}
\le
C\int_1^\infty \frac{1}{1+\lambda t^\alpha}\,dt.
\]
After the change of variables $s=\lambda^{1/\alpha}t$, we obtain
\[
\int_1^\infty \frac{1}{1+\lambda t^\alpha}\,dt
=
\lambda^{-1/\alpha}\int_{\lambda^{1/\alpha}}^\infty \frac{1}{1+s^\alpha}\,ds
\le
C\lambda^{-1/\alpha}.
\]
Hence,
\begin{equation}\label{noise-bound}
\mathbb{E}\Big[\sup_{v\in\mathbb{R}^n}\frac{(e,Av)^2}{\lla v\rra_\lambda^2}\Big]
\le
C\sigma^2\lambda^{-1/\alpha}.
\end{equation}

Finally, squaring \eqref{x4}, taking expectations, and using \eqref{noise-bound}, we arrive at
\[
\mathbb{E}\big[\lla x_\lambda-x^*\rra_\lambda^2\big]
\le
C\lambda \|x^*\|_W^2 + C\sigma^2\lambda^{-1/\alpha}.
\]
Since
\[
\lla x_\lambda-x^*\rra_\lambda^2
=
\lambda\|x_\lambda-x^*\|_W^2+\|Ax_\lambda-Ax^*\|^2,
\]
the estimates \eqref{p5}--\eqref{p6} follow immediately after dividing by $n$.
\pproof
\end{proof}

Theorem \ref{thm:2.1} yields convergence in expectation for the output error $Ax_\lambda-Ax^*$, but only boundedness in expectation for the source error $x_\lambda-x^*$ in the $W$-norm. We next show that the source error is convergent in expectation in a weaker topology.

\begin{lemma}\label{lem:2.2}
Let Assumption \ref{Assumption1} be satisfied, and define
\[
B=(W^{-1/2}A^TAW^{-1/2})^{1/4}W^{1/2}.
\]
Then
\[
\mathbb{E} \Big[\frac{1}{n}\|B(x_{\lambda}-x^*)\|^2\Big]
\leq C \lambda^{1/2} \frac{1}{n}\|x^*\|^2_{W}
+ \frac{C\sigma^2}{n\lambda^{1/2+1/\alpha}}.
\]
\end{lemma}

\begin{proof}
By Lemma \ref{lemma:eigv}, every $u\in\mathbb{R}^n$ admits the expansion
\[
u=\sum_{k=1}^n u_k\psi_k,
\qquad
u_k=(u,\psi_k)_W,
\]
and
\[
\|u\|_W^2=\sum_{k=1}^n u_k^2,
\qquad
\|Au\|^2=\sum_{k=1}^n \rho_k u_k^2.
\]
Moreover,
\[
B\psi_k=\rho_k^{1/4}W^{1/2}\psi_k,
\]
and hence
\[
\|Bu\|^2=\sum_{k=1}^n \rho_k^{1/2}u_k^2
\le
\Big(\sum_{k=1}^n u_k^2\Big)^{1/2}
\Big(\sum_{k=1}^n \rho_k u_k^2\Big)^{1/2}
=
\|u\|_W\|Au\|.
\]
Taking $u=x^*-x_\lambda$, we obtain
\begin{equation}\label{ccc0}
\|B(x^*-x_\lambda)\|^2
\le
\|Ax^*-Ax_\lambda\|\,\|x^*-x_\lambda\|_W.
\end{equation}
By Young's inequality,
\begin{equation}\label{B_eq}
\|B(x^*-x_\lambda)\|^2
\le
C\lambda^{1/2}\|x^*-x_\lambda\|_W^2
+
C\lambda^{-1/2}\|Ax^*-Ax_\lambda\|^2.
\end{equation}
Taking expectations in \eqref{B_eq} and applying \eqref{p5}--\eqref{p6} proves the result.
\end{proof}\pproof

\begin{rem}
Combining Theorem \ref{thm:2.1} and Lemma \ref{lem:2.2}, we obtain the a priori parameter choice
\begin{equation}\label{opt-parameter_w}
  \lambda^{\frac{1}{2}+\frac{1}{2\alpha}}= O\!\left(\sigma n^{-1/2}\Big(\frac{1}{\sqrt n}\|x^*\|_W\Big)^{-1}\right).
\end{equation}
With this choice, the corresponding error bounds become
\begin{align}
\mathbb{E} \Big[\frac{1}{\sqrt{n}}\|Ax_{\lambda}-Ax^*\| \Big]
&\leq C \lambda^{1/2} \frac{1}{\sqrt{n}}\|x^*\|_{W}, \label{eq:bound_A} \\
\mathbb{E} \Big[\frac{1}{\sqrt{n}}\|B(x_{\lambda}-x^*)\|\Big]
&\leq C \lambda^{1/4} \frac{1}{\sqrt{n}}\|x^*\|_{W}. \label{eq:bound_B}
\end{align}
\end{rem}

\section{Stochastic convergence for noisy data being sub-gaussian random variables}\label{sec:gaussian}

In this section, we turn to the case {\bf (M2)} for the data model \eqref{eq:data}, namely,
\begin{equation}\label{e1}
\mathbb E\Big[ \m{exp}(\lambda (e_i - \mathbb E[e_i]))\Big] \le \m{exp} \Big(\frac 12 \sigma^2 \lambda^2\Big) 
\qquad \forall\,\lambda \in \mathbb R.
\end{equation}
Under this assumption, we study the stochastic convergence behavior of the errors $\frac{1}{\sqrt{n}}\|Ax^*-Ax_{\lambda}\|$ and $\frac{1}{\sqrt{n}}\|Bx^*-Bx_{\lambda}\|$.

To facilitate the subsequent analysis, we first recall several basic notions and auxiliary results from the theory of sub-Gaussian random variables and empirical processes; see \cite{Chen-Zhang2018,Geer2000,Vaart1996} for further details and general background. A fundamental feature of a sub-Gaussian random variable is that its tail decays exponentially fast. More precisely, if $Z$ is sub-Gaussian, then its distribution satisfies
\begin{equation}\label{gg1}
\mathbb{P}(|Z-\E[Z]|\ge z)\le 2\,\m{exp}\Big(-\frac {z^2}{2\sigma^2}\Big)
\qquad \forall\, z>0.
\end{equation}
This tail bound plays a central role in deriving concentration inequalities for the stochastic quantities appearing in our analysis.

We shall also make use of Orlicz norms. Let $\psi$ be a nondecreasing convex function satisfying $\psi(0)=0$. Then the Orlicz norm of a random variable $Z$ is defined by
\begin{equation}\label{e2}
\|Z\|_\psi=\inf\Big\{C>0:\mathbb{E}\Big[\psi\Big(\frac{|Z|}{C}\Big)\Big]\le 1\Big\}.
\end{equation}
In particular, throughout most of the analysis we work with the $\psi_2$-Orlicz norm, corresponding to the function  $\psi_2(t)=e^{t^2}-1$ for $t>0$.  This norm provides a convenient way to quantify sub-Gaussian behavior. In particular, one has the estimate (see, for example, \cite[(4.5)]{Chen-Zhang2018})
\begin{equation}\label{e3}
\mathbb{P}(|Z|\ge z)\le 2\,\m{exp}\Big(-\frac{z^2}{\|Z\|_{\psi_2}^2}\Big)
\qquad \forall\, z>0.
\end{equation}

Next, let $\mathbb{T}$ be a semi-metric space endowed with a semi-metric $\sd$, and let $\{Z_t:t\in \mathbb{T}\}$ be a stochastic process indexed by $\mathbb{T}$. The process $\{Z_t:t\in \mathbb{T}\}$ is said to be sub-Gaussian if its increments satisfy
\begin{equation}\label{e41}
\mathbb{P}(|Z_s-Z_t|>z)\le 2\,\m{exp}\Big( -\frac{z^2}{2\,\sd(s,t)^2} \Big)
\qquad \forall \,s,t\in \mathbb{T},\ \forall\, z>0.
\end{equation}
This increment condition is the natural analogue of sub-Gaussianity for stochastic processes and is particularly useful in conjunction with entropy estimates.

For a semi-metric space $(\mathbb{T},\sd)$ and $\vep>0$, the covering number $N(\vep,\mathbb{T},\sd)$ is defined as the minimal number of $\vep$-balls required to cover $\mathbb{T}$. Its logarithm,
$\log N(\vep,\mathbb{T},\sd)$, is called the covering entropy. This quantity serves as a standard measure of the complexity of the index set $\mathbb{T}$ and will enter our estimates through a maximal inequality for sub-Gaussian processes.

More precisely, we shall use the following result from \cite[Section 2.2.1]{Vaart1996}.
\begin{lemma}\label{lem:3.1}
If $\{Z_t:t\in \mathbb{T}\}$ is a separable sub-Gaussian random process, then there exists a constant $K>0$ such that
\begin{equation*}
\|\sup_{s,t\in\mathbb{T}}|Z_s-Z_t|\|_{\psi_2}\le K\int^{\diam\, \mathbb{T}}_0\sqrt{\log N\Big(\frac\vep 2,\mathbb{T},\sd\Big)}\ d\vep\,.
\end{equation*}
\end{lemma}

In addition, we shall also use the following two auxiliary results from \cite{Chen-Zhang2018}.
\begin{lemma}\label{lem:3.2}
The family of random variables $\{E(u):=(e,Au):u\in \mathbb{R}^n\}$ forms a sub-Gaussian process with respect to 
\[
\sd(u,v)=\sigma \|Au-Av\|,
\qquad u,v\in\mathbb{R}^n.
\]
\end{lemma}

\begin{proof}
For \(u,v\in\mathbb R^n\), set \(a=A(u-v)\). Then
\[
E(u)-E(v)=(e,a)=\sum_{i=1}^n e_i a_i .
\]
By independence and the sub-Gaussian assumption, for every \(\lambda\in\mathbb R\),
\[
\mathbb E\exp\{\lambda(E(u)-E(v))\}
=
\prod_{i=1}^n \mathbb E\exp(\lambda a_i e_i)
\le
\prod_{i=1}^n
\exp\left(\frac12\sigma^2\lambda^2a_i^2\right).
\]
Hence
\[
\mathbb E\exp\{\lambda(E(u)-E(v))\}
\le
\exp\left(
\frac12\sigma^2\lambda^2\|A(u-v)\|^2
\right).
\]
Thus \(E(u)-E(v)\) is sub-Gaussian with variance proxy
\(\sigma^2\|Au-Av\|^2\), which gives the asserted semi-distance
\[
\sd(u,v)=\sigma\|Au-Av\|.
\]
\end{proof}\pproof

\begin{lemma}\label{lem:3.3}
Let $C_1>0$ and $K_1>0$ be constants, and let $Z$ be a random variable satisfying
\begin{equation*}
    \mathbb{P}(|Z|>\beta (1+z))\le C_1\,{\rm exp} \Big(-\frac{z^2}{K_1^2}\Big)
\qquad \forall\,\beta>0,\ \forall\, z\ge 1.
\end{equation*}
Then there exists a constant $C(C_1,K_1)>0$, depending only on $C_1$ and $K_1$, such that
\[
\|Z\|_{\psi_2}\le C(C_1,K_1)\,\beta.
\]
\end{lemma}

\begin{lemma}\label{lem:entropy-linear-alg}
Let Assumption \ref{Assumption1} hold, and define
\(B:=\{Au:\|u\|_W\le 1\}.\)
Then there exists a constant \(C>0\), independent of \(\varepsilon\), such that
\begin{equation}\label{entropy-linear-alg}
\log N\bigl(\varepsilon,B,\|\cdot\|\bigr)\le C\varepsilon^{-2/\alpha}
\end{equation}
for all sufficiently small \(\varepsilon>0\).
\end{lemma}

\begin{proof}
Let \(\{\psi_k\}_{k=1}^n\) be a \(W\)-orthonormal family of eigenvectors associated with the generalized eigenvalue problem
\[
(A\psi,Au)=\rho(\psi,u)_W \qquad \forall\,u\in\mathbb{R}^n,
\]
so that
\[
(\psi_i,\psi_j)_W=\delta_{ij}, \qquad \rho_1\ge \rho_2\ge \cdots \ge \rho_n\ge 0.
\]
Then every \(u\in\mathbb{R}^n\) admits the expansion \(u=\sum_{k=1}^n u_k\psi_k,\) and \( u_k=(u,\psi_k)_W,\)
and hence
\[
\|u\|_W^2=\sum_{k=1}^n u_k^2.
\]
Moreover, by Lemma \ref{lemma:eigv},
\[
(A\psi_i,A\psi_j)=\rho_i(\psi_i,\psi_j)_W=\rho_i\delta_{ij}.
\]
Therefore, for each \(k\) with \(\rho_k>0\), the vectors \(\phi_k:=\rho_k^{-1/2}A\psi_k\) form an orthonormal system in \(\mathbb{R}^n\) with respect to the Euclidean inner product. Consequently,
\[
Au=\sum_{\rho_k>0}\sqrt{\rho_k}\,u_k\,\phi_k,
\qquad
\|Au\|^2=\sum_{\rho_k>0}\rho_k u_k^2.
\]
Denoting \(z_k:=\sqrt{\rho_k}\,u_k\), the set \(B\) can be represented as the ellipsoid
\[
B=
\left\{
\sum_{\rho_k>0} z_k\phi_k:
\sum_{\rho_k>0}\frac{z_k^2}{\rho_k}\le 1
\right\}.
\]

Now let \(C_0>0\) be the constant in Assumption \ref{Assumption1}, so that
\[
\rho_k\le C_0 k^{-\alpha}, \qquad k=1,\dots,n.
\]
Define
\[
\widetilde{\rho}_k:=C_0 k^{-\alpha},
\qquad
a_k:=\widetilde{\rho}_k^{1/2}=\sqrt{C_0}\,k^{-\alpha/2}.
\]
Since \(\rho_k\le \widetilde{\rho}_k\), we have the inclusion
\[
B\subset \widetilde{B}:=
\left\{
\sum_{k=1}^n z_k\phi_k:
\sum_{k=1}^n \frac{z_k^2}{\widetilde{\rho}_k}\le 1
\right\},
\]
and therefore
\[
N(\varepsilon,B,\|\cdot\|)\le N(\varepsilon,\widetilde{B},\|\cdot\|).
\]
Thus it suffices to estimate the covering number of \(\widetilde{B}\).

Fix \(\varepsilon>0\) sufficiently small, and choose \(m\in\{0,1,\dots,n\}\) such that \(a_{m+1}<\frac{\varepsilon}{2}\le a_m,\) with the convention \(a_{n+1}=0\) if \(m=n\). Let \(P_m\) denote the orthogonal projection onto
\(\mathrm{span}\{\phi_1,\dots,\phi_m\}\). For any \(x=\sum_{k=1}^n z_k\phi_k\in \widetilde{B},\) we have
\[
\sum_{k=1}^n \frac{z_k^2}{a_k^2}\le 1.
\]
Hence
\[
\|(I-P_m)x\|^2
=\sum_{k=m+1}^n z_k^2
\le a_{m+1}^2\sum_{k=m+1}^n \frac{z_k^2}{a_k^2}
\le a_{m+1}^2
<
\frac{\varepsilon^2}{4}.
\]
Thus every point of \(\widetilde{B}\) lies within distance \(\varepsilon/2\) of the truncated ellipsoid
\(P_m\widetilde{B}\), and therefore
\[
N(\varepsilon,\widetilde{B},\|\cdot\|)
\le
N\!\left(\frac{\varepsilon}{2},P_m\widetilde{B},\|\cdot\|\right).
\]

The set \(P_m\widetilde{B}\) is the \(m\)-dimensional ellipsoid with semi-axis lengths
\(a_1,\dots,a_m\). By a standard volume comparison argument,
\[
N\!\left(\delta,P_m\widetilde{B},\|\cdot\|\right)
\le
\prod_{k=1}^m\left(1+\frac{2a_k}{\delta}\right),
\qquad \delta>0.
\]
Taking \(\delta=\varepsilon/2\), we obtain
\[
N(\varepsilon,\widetilde{B},\|\cdot\|)
\le
\prod_{k=1}^m\left(1+\frac{4a_k}{\varepsilon}\right).
\]
Since \(a_k\ge a_m\ge \varepsilon/2\) for \(1\le k\le m\), it follows that \(1+\frac{4a_k}{\varepsilon}\le 5\,\frac{a_k}{\varepsilon},\)
and hence
\[
N(\varepsilon,\widetilde{B},\|\cdot\|)
\le
5^m\prod_{k=1}^m \frac{a_k}{\varepsilon}.
\]
Using \(a_k=\sqrt{C_0}\,k^{-\alpha/2}\) and \(\varepsilon/2\le a_m=\sqrt{C_0}\,m^{-\alpha/2}\), we find \(\varepsilon^{-1}\le C\,m^{\alpha/2},\)
so that
\[
\prod_{k=1}^m \frac{a_k}{\varepsilon}
=
\prod_{k=1}^m \frac{\sqrt{C_0}\,k^{-\alpha/2}}{\varepsilon}
\le
C^m \prod_{k=1}^m \left(\frac{m}{k}\right)^{\alpha/2}
=
C^m\left(\frac{m^m}{m!}\right)^{\alpha/2}.
\]
By Stirling's formula,
\[
m!\sim \sqrt{2\pi m}\left(\frac{m}{e}\right)^m,
\]
and therefore
\[
\left(\frac{m^m}{m!}\right)^{\alpha/2}\le e^{Cm}.
\]
Consequently,
\[
\log N(\varepsilon,\widetilde{B},\|\cdot\|)\le Cm.
\]

Finally, from \(a_{m+1}<\varepsilon/2\) we obtain \(\sqrt{C_0}(m+1)^{-\alpha/2}<\frac{\varepsilon}{2},\) which implies
\[
m\le C\varepsilon^{-2/\alpha}.
\]
Combining the above estimates yields
\[
\log N(\varepsilon,B,\|\cdot\|)
\le
\log N(\varepsilon,\widetilde{B},\|\cdot\|)
\le
C\varepsilon^{-2/\alpha}.
\]
This completes the proof.
\end{proof}\pproof

\begin{theorem}\label{thm:3.1}
Let \( \rho_0=\frac{1}{\sqrt n}\|x^*\|_{W}+\sigma n^{-1/2}, \) and let $x_{\lambda}\in \mathbb{R}^n$ be the solution of \eqref{p1}. If 
$\lambda^{\frac{1}{2}+\frac{1}{2\alpha}}=O(\sigma n^{-1/2}\rho_0^{-1}),$
then there exists a constant $C>0$, independent of
\(n,\lambda,\sigma\), such that
\[
\mathbb{P}\!\left(\frac{1}{\sqrt n}\|Ax_{\lambda}-Ax^*\| \ge \lambda^{1/2}\rho_0 z\right)\le 2e^{-Cz^2},
\qquad
\mathbb{P}\!\left(\frac{1}{\sqrt n}\|x_{\lambda}\|_{W}\ge \rho_0 z\right)\le 2e^{-Cz^2}.
\]
\end{theorem}

\begin{proof}
By \eqref{e3}, it suffices to prove
\begin{equation}\label{xx2}
\Big\|\frac{1}{\sqrt n}\|Ax_{\lambda}-Ax^*\|\Big\|_{\psi_2}\le C\lambda^{1/2}\rho_0,
\qquad
\Big\|\frac{1}{\sqrt n}\|x_{\lambda}\|_W\Big\|_{\psi_2}\le C\rho_0.
\end{equation}
Since the two estimates can be derived in the same manner, we prove only the first one by a peeling argument.

From \eqref{p1}, we have
\begin{equation}\label{g1}
\|Ax_{\lambda}-Ax^*\|^2+\lambda \|x_{\lambda}\|_{W}^2
\le 2(e,Ax_{\lambda}-Ax^*)+\lambda \|x^*\|_{W}^2.
\end{equation}
Moreover, by the definition of \(\rho_0\),
\begin{equation}\label{g1a}
\|x^*\|_W\le \sqrt n\,\rho_0,
\qquad
\lambda\|x^*\|_W^2\le \lambda n\rho_0^2.
\end{equation}

Let $\theta>0$ and $\rho>0$ be constants to be specified later, and define, for $i,j\ge 1$,
\begin{equation}\label{g2}
A_0=[0,\theta), \qquad A_i=[2^{i-1}\theta,2^i\theta),
\qquad
B_0=[0,\rho), \qquad B_j=[2^{j-1}\rho,2^j\rho).
\end{equation}
Choose \(\rho=2\sqrt n\,\rho_0 .\) For $i,j\ge 0$, set
\[
F_{ij}
=
\{v\in \mathbb{R}^n:\|Av-Ax^* \|\in A_i,\ \|v\|_W\in B_j\}.
\]
Then
\begin{equation}\label{g3}
\mathbb{P}(\|Ax_{\lambda}-Ax^*\|>\theta)
\le
\sum_{i=1}^{\infty}\sum_{j=0}^{\infty}
\mathbb{P}(x_{\lambda}\in F_{ij}).
\end{equation}

We next estimate $\mathbb{P}(x_{\lambda}\in F_{ij})$. For \(i,j\ge0\), define
\[
G_{ij}:=\bigcup_{k=0}^{i}\bigcup_{\ell=0}^{j}F_{k\ell}.
\]
Since \(\rho=2\sqrt n\,\rho_0\) and \(\|x^*\|_W\le \sqrt n\,\rho_0\), we have \(x^*\in F_{00}\subset G_{ij}.\)

By Lemma \ref{lem:3.2}, the process \(\{(e,Av):v\in \mathbb{R}^n\}\) is sub-Gaussian with respect to \(\sd(u,v)=\sigma\|Au-Av\|.\)
For any \(u,v\in G_{ij}\), by the definition of \(G_{ij}\),
\[
\|Au-Ax^*\|<2^i\theta,
\qquad
\|Av-Ax^*\|<2^i\theta.
\]
Therefore
\[
\|Au-Av\|
\le
\|Au-Ax^*\|+\|Av-Ax^*\|
<
2^{i+1}\theta.
\]
Consequently,
\begin{equation}\label{diamG}
\operatorname{diam}(G_{ij},\sd)
=
\sup_{u,v\in G_{ij}}\sigma\|Au-Av\|
\le
\sigma 2^{i+1}\theta .
\end{equation}

Using Lemma \(\ref{lem:3.1}\), \eqref{diamG}, and the definition of
\(\sd\), we obtain
\begin{align}
\Big\|\sup_{x\in F_{ij}} |(e,Ax-Ax^*)|\Big\|_{\psi_2}
&\le
\Big\|\sup_{x\in G_{ij}} |(e,Ax-Ax^*)|\Big\|_{\psi_2}  \notag \\
&\le
K\int_0^{\sigma \cdot 2^{i+1}\theta}
\sqrt{\log N\!\left(\frac{\vep}{2},G_{ij},\sd\right)}\,d\vep \notag \\
&=
K\int_0^{\sigma \cdot 2^{i+1}\theta}
\sqrt{\log N\!\left(\frac{\vep}{2\sigma},G_{ij},\|A\cdot\|\right)}\,d\vep.
\label{max1}
\end{align}
Since
\[
G_{ij}\subset \{v\in\mathbb R^n:\|v\|_W<2^j\rho\},
\]
we have
\[
A(G_{ij})
\subset
2^j\rho\,\{Au:\|u\|_W\le 1\}.
\]
Therefore, by Lemma \(\ref{lem:entropy-linear-alg}\),
\[
\log N\left(\frac{\varepsilon}{2\sigma},G_{ij},\|A\cdot\|\right)
\le
C\left(
\frac{2\sigma\,2^j\rho}{\varepsilon}
\right)^{2/\alpha}
\le
C\left(
\frac{\sigma\,2^j\rho}{\varepsilon}
\right)^{2/\alpha},
\]
where the factor \(2^{2/\alpha}\) has been absorbed into the generic
constant \(C\).

Substituting this into the previous estimate \eqref{max1} yields
\begin{align}
\Big\|\sup_{x\in F_{ij}} |(e,Ax-Ax^*)|\Big\|_{\psi_2}
&\le
K\sqrt{C}\int_0^{\sigma \cdot 2^{i+1}\theta}
\Big(\frac{\sigma 2^j\rho}{\vep}\Big)^{1/\alpha}\,d\vep \notag\\
&=
C\sigma (2^j\rho)^{1/\alpha}(2^i\theta)^{1-1/\alpha}.
\label{g5}
\end{align}
where we used \(\alpha>1\).

We now estimate \(\mathbb P(x_\lambda\in F_{ij})\). Suppose first that
\(i,j\ge1\). On the event \(x_\lambda\in F_{ij}\), \eqref{g1} and
\eqref{g1a} imply
\[
2^{2(i-1)}\theta^2+\lambda 2^{2(j-1)}\rho^2
\le
2\sup_{x\in F_{ij}} |(e,Ax-Ax^*)|
+\lambda n\rho_0^2 .
\]
Hence, by \eqref{g5} and the \(\psi_2\)-tail estimate \eqref{e3},
\begin{align}
\mathbb P(x_\lambda\in F_{ij})
&\le
2\exp\Bigg[
-\frac{1}{C\sigma^2}
\left(
\frac{
2^{2(i-1)}\theta^2
+\lambda 2^{2(j-1)}\rho^2
-\lambda n\rho_0^2
}{
(2^i\theta)^{1-1/\alpha}
(2^j\rho)^{1/\alpha}
}
\right)^2
\Bigg].
\label{pij1}
\end{align}

For $z\ge 1$, choose
\[
\theta^2=\lambda n\rho_0^2(1+z)^2,
\qquad
\rho=2\sqrt n\,\rho_0.
\]
Then, under the condition
\[
\lambda^{\frac12+\frac{1}{2\alpha}}=O(\sigma n^{-1/2}\rho_0^{-1}),
\]
a direct computation shows that for $i,j\ge 1$,
\begin{equation}\label{eq:Pij}
\mathbb{P}(x_{\lambda}\in F_{ij})
\le
2\exp\Bigg[
-C
\Bigg(
\frac{2^{2(i-1)}z(1+z)+2^{2(j-1)}}
{(2^i(1+z))^{1-1/\alpha}(2^j)^{1/\alpha}}
\Bigg)^2
\Bigg].
\end{equation}

To simplify this bound, we apply Young's inequality in the form
\[
ab\le \frac{a^p}{p}+\frac{b^q}{q},
\qquad
a,b>0,\quad p,q>1,\quad p^{-1}+q^{-1}=1,
\]
which gives
\[
(2^i(1+z))^{1-1/\alpha}(2^j)^{1/\alpha}
\le
C\big((1+z)2^i+2^j\big).
\]
Substituting this into \eqref{eq:Pij}, we arrive at
\[
\mathbb{P}(x_{\lambda}\in F_{ij})
\le
2\exp\big(-C(2^{2i}z^2+2^{2j})\big),
\qquad i,j\ge 1.
\]
The case \(i\ge1\), \(j=0\), is treated in the same way and gives
\[
\mathbb{P}(x_{\lambda}\in F_{i0})
\le
2\exp\big(-C2^{2i}z^2\big).
\]

Combining these bounds over all $i,j\ge 0$, and using the facts that for \(C \ge 1\),
\[
\sum_{j=1}^{\infty}e^{-C2^{2j}}\le e^{-C}<1,
\qquad
\sum_{i=1}^{\infty}e^{-C2^{2i}z^2}\le e^{-Cz^2},
\]
we obtain
\[
\sum_{i=1}^{\infty}\sum_{j=0}^{\infty}
\mathbb{P}(x_{\lambda}-x^*\in F_{ij})
\le
4e^{-Cz^2}.
\]
Thus, by \eqref{g3},
\begin{equation}\label{g7}
\mathbb{P}\big(\|Ax_{\lambda}-Ax^*\|>\lambda^{1/2}\sqrt n\,\rho_0(1+z)\big)
\le
4e^{-Cz^2},
\qquad z\ge 1.
\end{equation}
It then follows from Lemma \ref{lem:3.3} that
\[
\Big\|\frac{1}{\sqrt n}\|Ax_{\lambda}-Ax^*\|\Big\|_{\psi_2}
\le
C\lambda^{1/2}\rho_0,
\]
which is the first estimate in \eqref{xx2}. 

The \(W\)-norm estimate is obtained by the same peeling argument, now taking
\(\theta=\lambda^{1/2}\sqrt n\,\rho_0\) and
\(\rho=2\sqrt n\,\rho_0(1+z)\), which gives
\[
\mathbb P(\|x_\lambda\|_W>\rho)\le C e^{-cz^2},\qquad z\ge1,
\]
and therefore
\[
\left\|\frac1{\sqrt n}\|x_\lambda\|_W\right\|_{\psi_2}\le C\rho_0 .
\]
\end{proof}\pproof

Using the operator $B$ introduced in Lemma \ref{lem:2.2}, we can further derive a high-probability estimate in the corresponding weak topology.

\begin{cor}
With the notation of Theorem \ref{thm:3.1} and Lemma \ref{lem:2.2}, one has
\[
\mathbb{P}\!\left(
\frac{1}{\sqrt n}\|B(x_{\lambda}-x^*)\|
\ge
\lambda^{1/4}\rho_0 z
\right)
\le
2e^{-Cz^2}.
\]
\end{cor}

\begin{proof}
By \eqref{B_eq},
\[
\|Bx^*-Bx_{\lambda}\|^2
\le
C\lambda^{1/2}\|x^*-x_{\lambda}\|_W^2
+
C\lambda^{-1/2}\|Ax^*-Ax_{\lambda}\|^2.
\]
The asserted estimate is therefore an immediate consequence of Theorem \ref{thm:3.1}.
\end{proof}\pproof

\begin{rem}
The above results suggest the a priori parameter choice
\begin{equation}\label{opt-parameter}
\lambda^{\frac{1}{2}+\frac{1}{2\alpha}}
=
O\!\left(
\sigma n^{-1/2}
\left(\frac{1}{\sqrt n}\|x^*\|_{W}+\sigma n^{-1/2}\right)^{-1}
\right).
\end{equation}
Under this choice, one obtains the high-probability bounds
\begin{align}
\mathbb{P}\!\left(
\frac{1}{\sqrt n}\|Ax_{\lambda}-Ax^*\|
\ge \lambda^{1/2}\rho_0 z
\right)
&\le 2e^{-Cz^2}, \label{eq:Pb-A} \\
\mathbb{P}\!\left(
\frac{1}{\sqrt n}\|B(x_{\lambda}-x^*)\|
\ge \lambda^{1/4}\rho_0 z
\right)
&\le 2e^{-Cz^2}. \label{eq:Pb-B}
\end{align}
\end{rem}

\section{Adaptive method to find the optimal regularization parameter}\label{sec:adaptive}

In practical applications, neither the exact solution nor, in some cases, the noise strength \(\sigma\) is available. Consequently, the a priori parameter-choice rule \eqref{opt-parameter_w} is generally not directly implementable. We also note that \eqref{opt-parameter_w} may be regarded as a practical approximation to the prior parameter-choice rule \eqref{opt-parameter} in Section \ref{sec:gaussian}, since the two differ only by the additive term \(\sigma n^{-1/2}\) in the denominator, which is typically negligible when \(n\) is large. Therefore, the two choices are asymptotically close. Motivated by this observation, we approximate the unknown solution norm \(\|x^*\|\) by \(\|x_\lambda\|\), and estimate the noise strength \(\sigma\) by the empirical residual \(\|Ax_{\lambda}-b\|\). This leads to the following adaptive procedure for selecting the regularization parameter using only the observed data.

\begin{algorithm}[!ht]
\caption{An adaptive method to find the optimal regularization parameter.}
\label{algo}
\begin{algorithmic}[1]
\Require Observation data $b$, number of observations $n$, tolerance $tol_{\lambda}$
\Ensure An approximate optimal regularization parameter $\lambda_k$ and the corresponding regularized solution $x_{\lambda_k}$

\State Set $k \gets 0$ and $\lambda_{-1} \gets +\infty$
\State Initialize the parameter $\lambda_0$ according to
\[
\lambda_0^{\frac{1}{2}+\frac{1}{2\alpha}} \gets n^{-1/2}
\]
\State Compute the initial regularized solution $x_{\lambda_0}$ by solving \eqref{p1} with $\lambda=\lambda_0$

\While{$|\lambda_k-\lambda_{k-1}| > tol_{\lambda}$}
    \State Given the current parameter $\lambda_k$, solve \eqref{p1} to obtain the regularized solution $x_{\lambda_k}$
    \State Estimate the next parameter by 
    \[
    \lambda_{k+1}^{\frac{1}{2}+\frac{1}{2\alpha}}
    \gets
    C\left(\frac{1}{\sqrt n}\|Ax_{\lambda_k}-b\|\right)
    n^{-1/2}
    \left(\frac{1}{\sqrt n}\|x_{\lambda_k}\|_{W}\right)^{-1}
    \]
    where $C > 0$ is a user–specified constant
    \State Update the iteration index: $k \gets k+1$
\EndWhile

\State \Return Parameter $\lambda_k$ and the corresponding regularized solution $x_{\lambda_k}$
\end{algorithmic}
\end{algorithm}

The proposed scheme is entirely data-driven and can be implemented in a straightforward manner. It starts from an explicit initial guess and then refines the parameter through a fixed-point type iteration until convergence. Since each update depends only on the current residual and the norm of the current regularized solution, the method is well suited for practical inverse problems in which the underlying noise characteristics are unavailable or difficult to estimate reliably.
\section{Numerical results}\label{sec:num}
In this section, we present several numerical experiments to complement the preceding theoretical analysis. Our main objective is to examine whether the observed numerical behavior is consistent with the theoretical predictions, including the near-optimality of the a priori parameter choice, the boundedness suggested by the error estimates, and the effectiveness of the adaptive parameter-selection algorithm. The noisy data are generated by
\[
b=y+\sigma \xi,
\]
where \(\xi=(\xi_i)_{i=1}^n\) consists of independent standard Gaussian random variables and
\[
\sigma=\delta \frac{1}{\sqrt{n}}\|Ax^\ast\|,
\]
with \(\delta>0\) denoting the prescribed relative noise level and \(\sigma\) representing the corresponding noise strength.

\begin{example}[A prototypical linear inverse problem]\label{nr:int}
We consider a prototypical linear inverse problem induced by a continuous kernel \(\kappa:\overline{\Omega}\times\overline{\Omega}\to\mathbb{R}
\),\quad \(\Omega=(0,1)\).
The associated integral operator \(K:L^{2}(\Omega)\to L^{2}(\Omega)\) is defined by
\begin{equation}
(Kx)(t)=\int_{\Omega}\kappa(t,s)x(s)\,ds.
\end{equation}
Since the kernel \(\kappa\) is continuous on the compact set \(\overline{\Omega}\times\overline{\Omega}\), the operator \(K\) is compact on \(L^{2}(\Omega)\). Its adjoint operator \(K^{*}:L^{2}(\Omega)\to L^{2}(\Omega)\) is given by \((T^{*}y)(s)=\int_{\Omega}\kappa(t,s)y(t)\,dt.\)

To obtain a finite-dimensional approximation, we partition the interval \(\overline{\Omega}\) into \(n\) uniform subintervals \(\Big[\frac{i-1}{n},\frac{i}{n}\Big], i=1,\dots,n,\) and apply a quadrature rule to discretize the integral equation. This yields the linear system
\[
Ax^{*}=y,
\]
where
\[
A=\frac{1}{n}\Bigl(\kappa\Bigl(\frac{j-1}{n},\frac{2i-1}{2n}\Bigr)\Bigr)_{j,i=1}^{n},
\qquad
x^{*}=\Bigl(x\Bigl(\frac{2j-1}{2n}\Bigr)\Bigr)_{j=1}^{n}.
\]

In the present example, we choose the kernel function and the exact solution as
\begin{equation}\label{kernel}
  \kappa(t,s)=
\begin{cases}
s(1-t), & s\le t,\\
t(1-s), & s\ge t,
\end{cases}
\end{equation}
and \(x(t)=-6t^{2}(1-t)(2-8t+7t^{2}),\) respectively, and set the weight matrix to \(W = I\). It is well known that the kernel \eqref{kernel} is the Green’s function of the one-dimensional Dirichlet Laplacian on \((0,1)\). Hence, the associated operator \(K\) has singular values satisfying \(s_j(K)\sim j^{-2}\). Therefore, for a standard discretization \(A\) of \(K\), the generalized eigenvalues \(\rho_k\) of \eqref{ccc} behave like the squared singular values of \(A\), so that \(\rho_k \sim k^{-4}.\) It is natural to take \(\alpha=4\) in Assumption \eqref{Assumption1} for this example.
\end{example}

\begin{itemize}
    \item Near-optimality of the a priori parameter choice (\ref{opt-parameter}) in terms of the output error \(\frac{1}{\sqrt{n}}\|Ax_{\lambda}-Ax^*\|\).
\end{itemize}
  With the noise level fixed at \(\delta=0.01\), we choose ten logarithmically equispaced values of \(\lambda\) in the interval \([10^{-10},10^{-4}]\). For each \(\lambda\), we compute the regularized solution \(x_\lambda\) and evaluate the corresponding output error \(\frac{1}{\sqrt n}\|Ax_\lambda-Ax^*\|.\) These values are then compared with the parameter predicted by rule \eqref{opt-parameter}, denoted by \(\lambda_{\mathrm{pred}}\), where the implicit constant in the \(O(\cdot)\) term is set equal to one. The same experiment is repeated for several discretization levels \(n\) in order to assess whether the proposed parameter choice remains effective across different problem sizes. As shown in Figure \ref{fig:parameter}, the parameter predicted by rule \eqref{opt-parameter} is nearly optimal, in the sense that it lies close to the minimizer of the output error.

\begin{figure}[hbt!]
	\centering
	\setlength{\tabcolsep}{0pt}
	\begin{tabular}{cc}
		\includegraphics[width=0.3\textwidth]{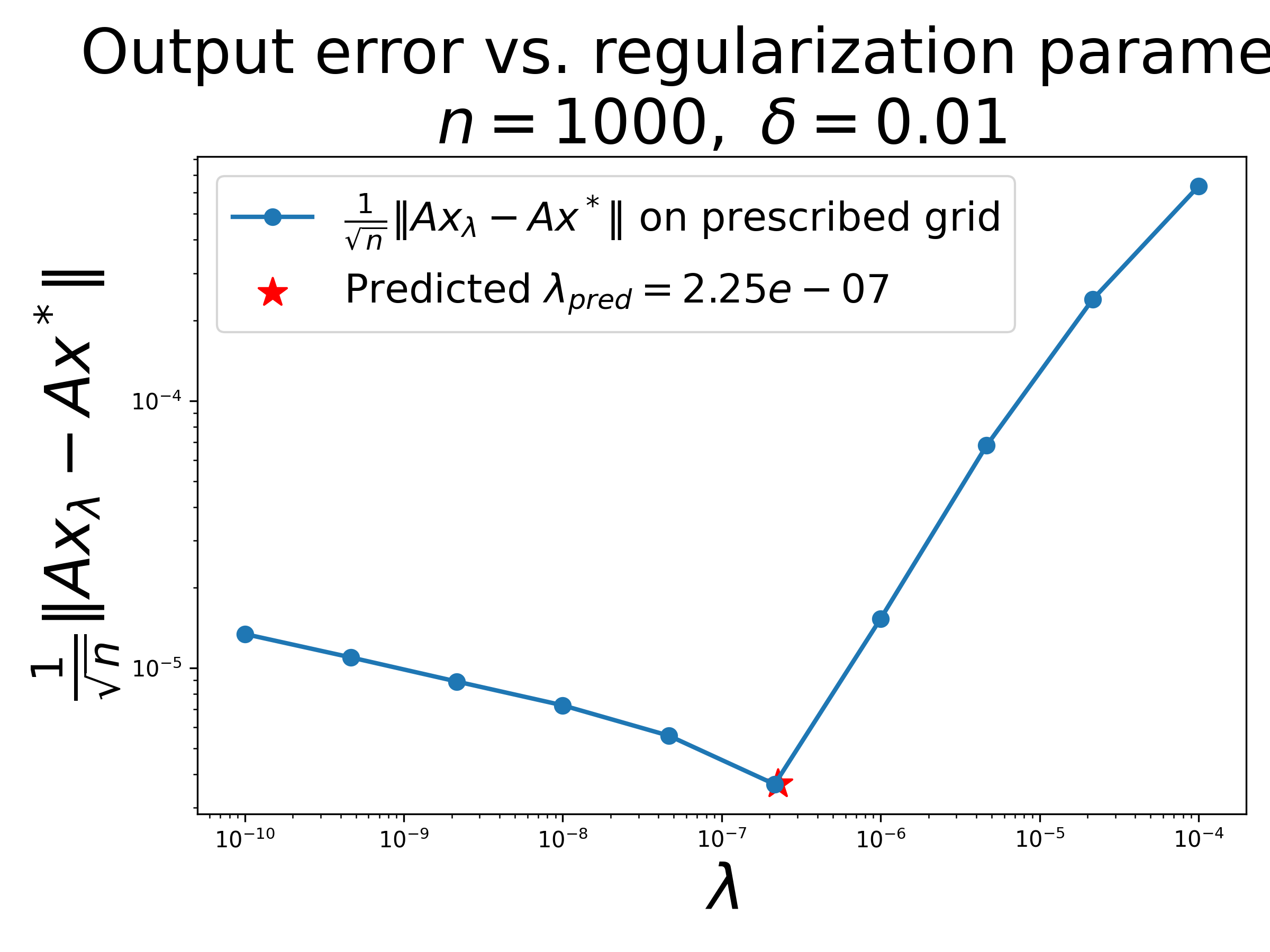} 
        \includegraphics[width=0.3\textwidth]{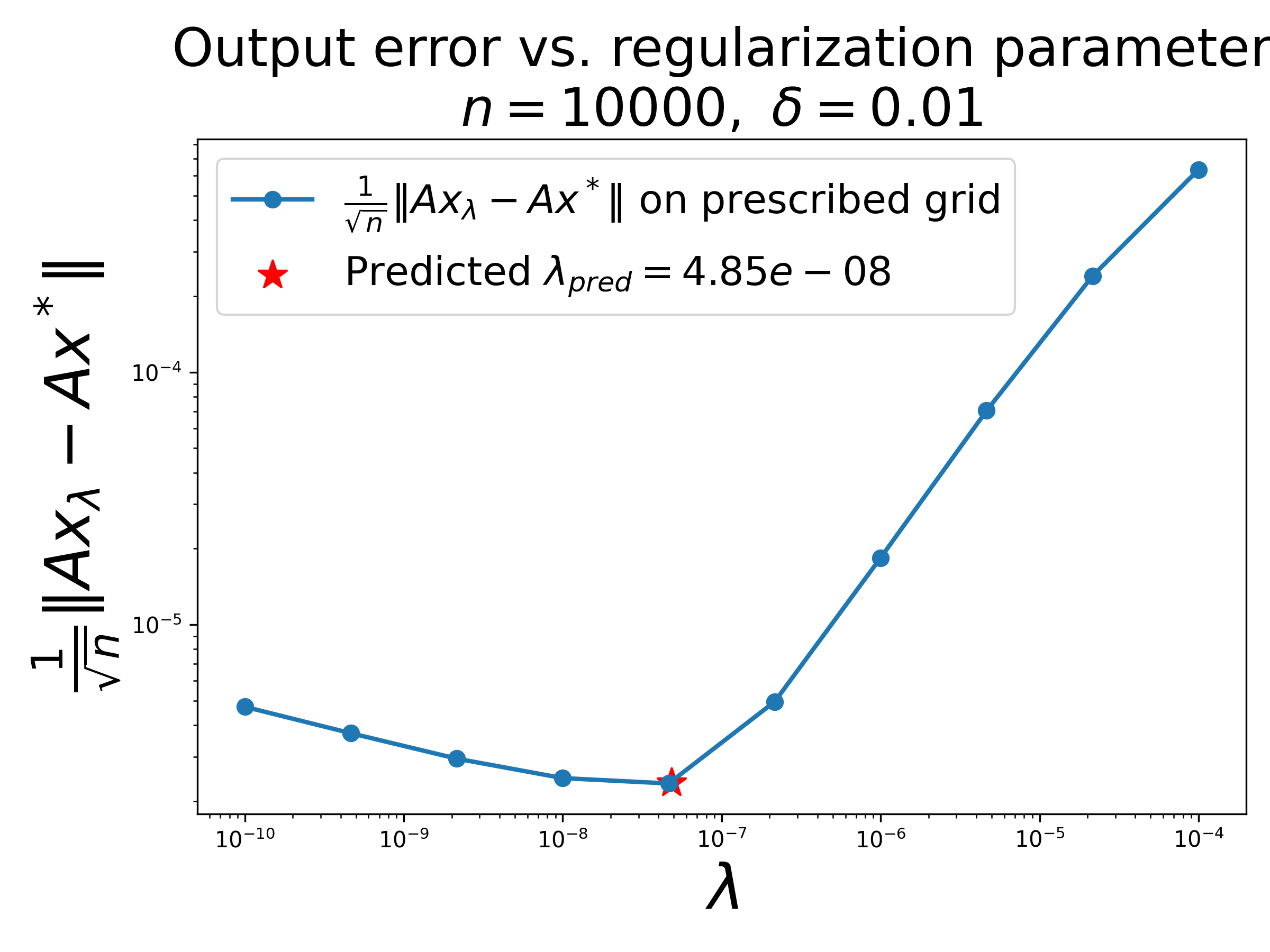}
        \includegraphics[width=0.3\textwidth]{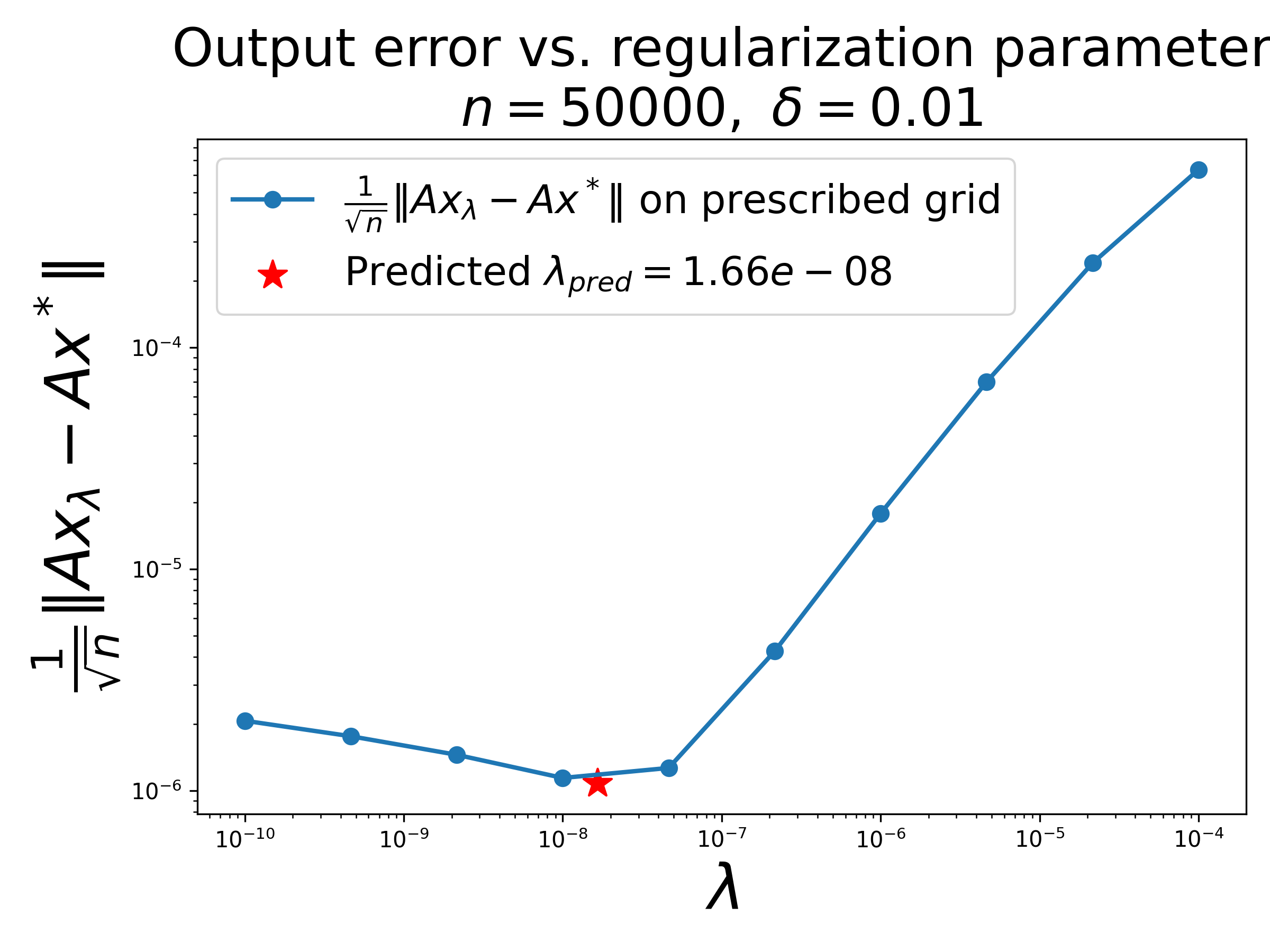} 
	\end{tabular}
	\caption{Output error \(\frac{1}{\sqrt{n}}\|Ax_{\lambda}-Ax^*\|\) versus the regularization parameter \(\lambda\) on a log--log scale for different discretization levels \(n\), illustrating the near-optimality of the parameter choice rule \eqref{opt-parameter}.}
	\label{fig:parameter}
\end{figure}

\begin{itemize}
    \item Numerically illustrate the boundedness predicted by \eqref{eq:bound_A} and \eqref{eq:bound_B}.
\end{itemize}

To examine the scaling predicted by the expectation bounds, we performed a Monte Carlo experiment for the discretization levels \(n\in\{1000, 2000, 4000, 8000, 16000\}\) and the noise levels \(\delta\in\{10^{-1},10^{-2},10^{-3},10^{-4}\}.\) For each pair \((n,\delta)\), the regularization parameter \(\lambda\) was chosen according to the predicted rule \eqref{opt-parameter}, where the implicit constant in the \(O(\cdot)\) term is set equal to one. The corresponding regularized solution \(x_\lambda\) was then computed, and the expectations\(\mathbb{E}\!\left[n^{-1/2}\|Ax_\lambda-Ax^*\|\right] \) and \(\mathbb{E}\!\left[n^{-1/2}\|B(x_\lambda-x^*)\|\right] \) with \(B=(A^\top A)^{1/4}\)
were approximated by sample averages over 2000 independent Gaussian noise realizations. The empirical means were then plotted against \(\lambda\) on log--log scales and fitted by least squares in logarithmic coordinates.  As shown in Figure \ref{fig:expectation}, the fitted slopes are \(0.498\) and \(0.250\), respectively, which are in excellent agreement with the predicted exponents \(1/2\) and \(1/4\). This provides numerical evidence for the bounds \eqref{eq:bound_A} and \eqref{eq:bound_B}.

\begin{figure}[hbt!]
	\centering
	\setlength{\tabcolsep}{0pt}
	\begin{tabular}{cc}
		\includegraphics[width=0.4\textwidth]{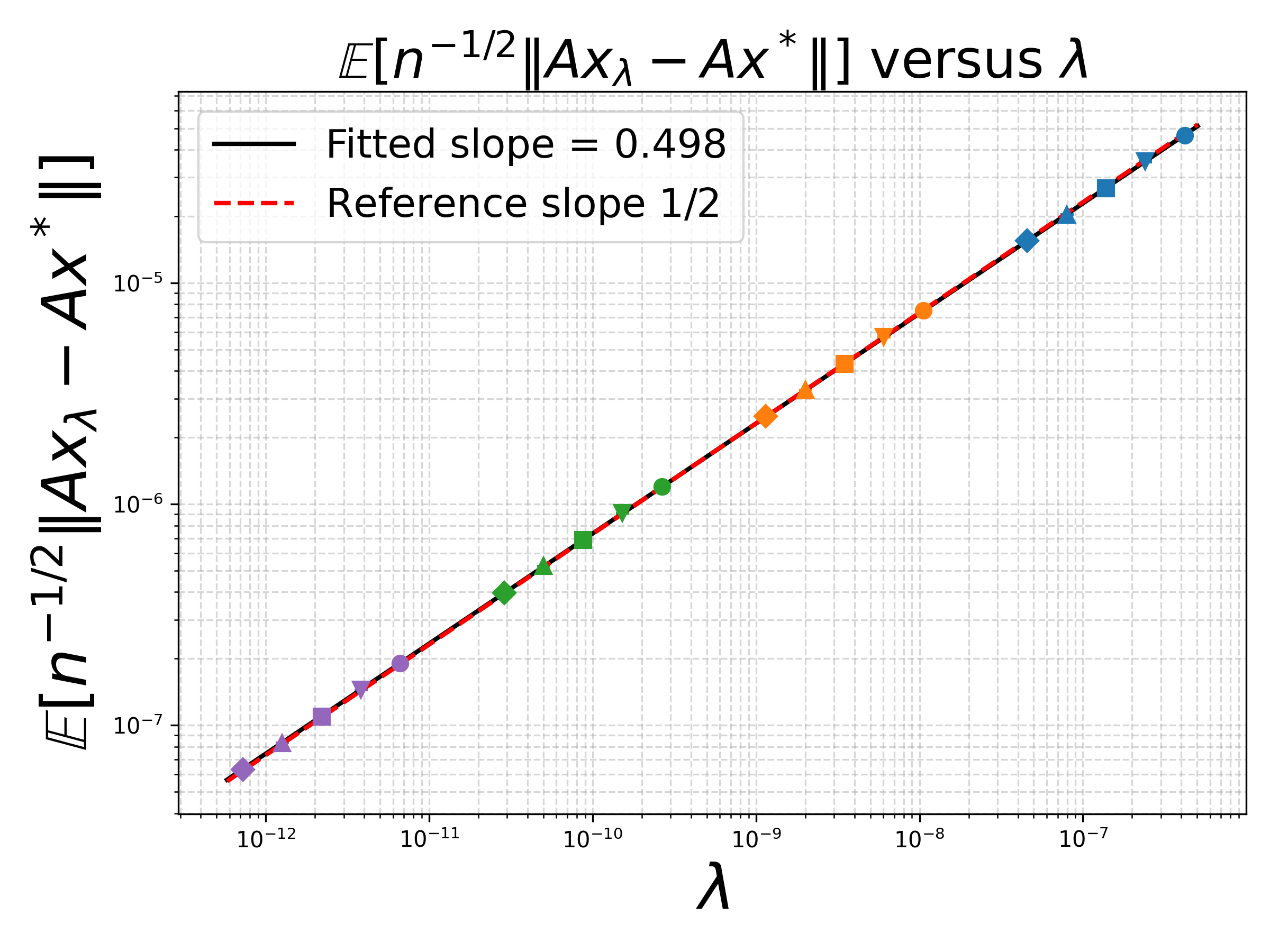} & \
		\includegraphics[width=0.4\textwidth]{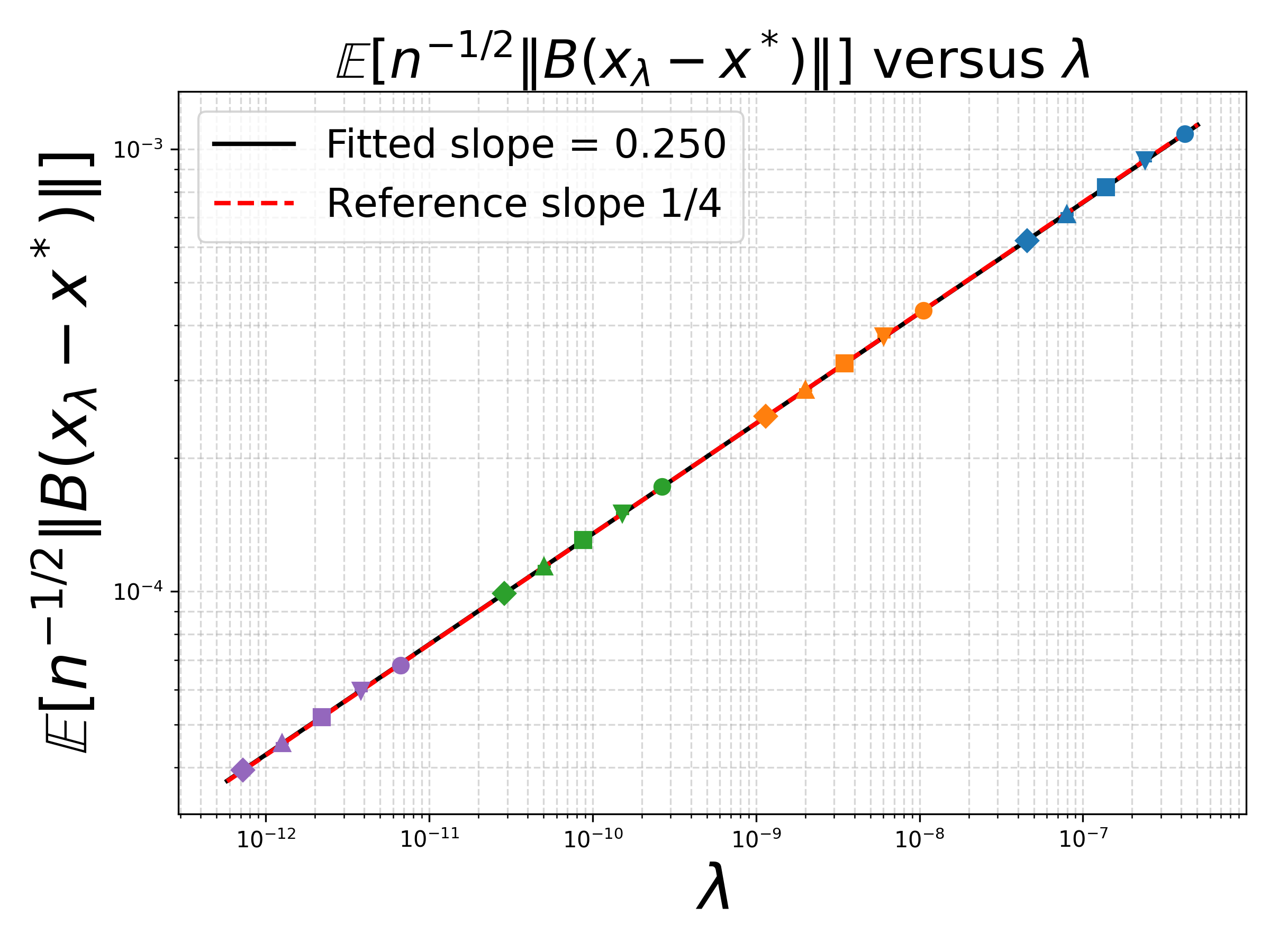} 
	\end{tabular}
	\caption{Log--log plots of \(\lambda\) versus \(\mathbb{E}[n^{-1/2}\|Ax_\lambda-Ax^*\|]\) (left) and \(\mathbb{E}[n^{-1/2}\|B(x_\lambda-x^*)\|]\) (right) with \(B=(A^\top A)^{1/4}\) for \(n\in\{1000, 2000, 4000, 8000, 16000\}\) and \(\delta\in\{10^{-1},10^{-2},10^{-3},10^{-4}\}\). The solid lines denote least-squares fits in logarithmic coordinates, with fitted slopes \(0.498\) and \(0.250\), while the dashed lines indicate the reference slopes \(1/2\) and \(1/4\), respectively. The observed agreement supports the predicted power-law behavior of the two expectation bounds.}
	\label{fig:expectation}
\end{figure}

\begin{itemize}
    \item Numerical verification of the effectiveness of Algorithm  \ref{algo}.
\end{itemize}
We consider the noise levels \(\delta\in\{0.1,0.01,0.001\}\) and the discretization sizes \(n\in\{2000,5000,10000\}\). Throughout the experiments, the stopping tolerance is set to \(\mathrm{tol}_\lambda=10^{-10}\), and Algorithm~\ref{algo} is employed to determine the regularization parameter. For each pair \((\delta,n)\), we report the final regularization parameter, the number of iterations, the relative reconstruction error \(\|x_{\lambda}-x^\ast\|/\|x^\ast\|,\) the relative error with respect to the exact data \(\|Ax_{\lambda}-Ax^*\|/\|Ax^*\|,\) and the relative residual with respect to the noisy data \(\|Ax_{\lambda}-b\|/\|b\|.\)

Table~\ref{tab:numerical-results} summarizes the numerical results for all tested noise levels and discretization sizes. For each \((\delta,n)\), it lists the noise strength \(\sigma\), the final regularization parameter \(\lambda_{\mathrm{final}}\), the number of iterations, the relative reconstruction error, and the relative data errors with respect to both the exact data \(y\) and the noisy data \(b\). The results show that \(\lambda_{\mathrm{final}}\) decreases as the noise level \(\delta\) decreases or the discretization size \(n\) increases. Meanwhile, both the relative reconstruction error \(\|x_{\lambda}-x^\ast\|/\|x^\ast\|,\) and the relative exact-data error \(\|Ax_\lambda-Ax^*\|/\|Ax^*\|\) become smaller for lower noise levels and finer discretizations, whereas the residual \(\|Ax_\lambda-b\|/\|b\|\) remains of the same order as the prescribed noise level \(\delta\). These results indicate that the proposed parameter-choice algorithm \ref{algo} produces stable regularization parameters and accurate reconstructions across all tested regimes.

Figure~\ref{fig:adaptive method} shows the convergence histories of the iterative sequence \(\lambda_k\). In all cases, \(\lambda_k\) decreases rapidly during the first few iterations and then stabilizes, which demonstrates the fast and stable convergence of the parameter-selection procedure. Figure~\ref{fig:recon_sol} compares the reconstructed solutions with the exact solution. It can be seen that the regularized solutions remain close to the exact solution for all tested values of \(n\), and that the reconstruction accuracy improves as the noise level \(\delta\) decreases and the discretization size \(n\) increases.

\begin{table}[htbp]
\centering
\caption{Numerical results for different noise levels $\delta$ and discretization sizes $n$.}
\label{tab:numerical-results}
\resizebox{\textwidth}{!}{%
\begin{tabular}{cccccccc}
\toprule
$\delta$ & $n$ & $\sigma$ & $\lambda_{\mathrm{final}}$ & iters & $\|x_{\lambda}-x^\ast\|/\|x^\ast\|$ & $\|Ax_{\lambda}-Ax^*\|/\|Ax^*\|$ & $\|Ax_{\lambda}-b\|/\|b\|$ \\
\midrule
0.1   & 2000  & $4.7117\times 10^{-4}$ & $3.188834\times 10^{-6}$ & 8 & $8.875881\times 10^{-2}$ & $8.543793\times 10^{-3}$ & $9.996602\times 10^{-2}$ \\
0.1   & 5000  & $4.7117\times 10^{-4}$ & $1.712181\times 10^{-6}$ & 7 & $7.155710\times 10^{-2}$ & $6.324914\times 10^{-3}$ & $9.960679\times 10^{-2}$ \\
0.1   & 10000 & $4.7117\times 10^{-4}$ & $1.072741\times 10^{-6}$ & 6 & $5.760931\times 10^{-2}$ & $5.193584\times 10^{-3}$ & $9.973962\times 10^{-2}$ \\
0.01  & 2000  & $4.7117\times 10^{-5}$ & $1.401075\times 10^{-7}$ & 7 & $3.691624\times 10^{-2}$ & $8.594145\times 10^{-4}$ & $9.894290\times 10^{-3}$ \\
0.01  & 5000  & $4.7117\times 10^{-5}$ & $7.613366\times 10^{-8}$ & 6 & $3.694457\times 10^{-2}$ & $7.594605\times 10^{-4}$ & $9.896408\times 10^{-3}$ \\
0.01  & 10000 & $4.7117\times 10^{-5}$ & $4.892315\times 10^{-8}$ & 5 & $2.807365\times 10^{-2}$ & $4.815258\times 10^{-4}$ & $1.005086\times 10^{-2}$ \\
0.001 & 2000  & $4.7117\times 10^{-6}$ & $6.494814\times 10^{-9}$ & 7 & $2.905355\times 10^{-2}$ & $1.347285\times 10^{-4}$ & $9.900827\times 10^{-4}$ \\
0.001 & 5000  & $4.7117\times 10^{-6}$ & $3.510672\times 10^{-9}$ & 6 & $2.880116\times 10^{-2}$ & $9.333476\times 10^{-5}$ & $9.869945\times 10^{-4}$ \\
0.001 & 10000 & $4.7117\times 10^{-6}$ & $2.214830\times 10^{-9}$ & 5 & $2.291959\times 10^{-2}$ & $5.399370\times 10^{-5}$ & $9.879263\times 10^{-4}$ \\
\bottomrule
\end{tabular}%
}
\end{table}

\begin{figure}[hbt!]
	\centering
	\setlength{\tabcolsep}{0pt}
	\begin{tabular}{cc}
		\includegraphics[width=0.48\textwidth]{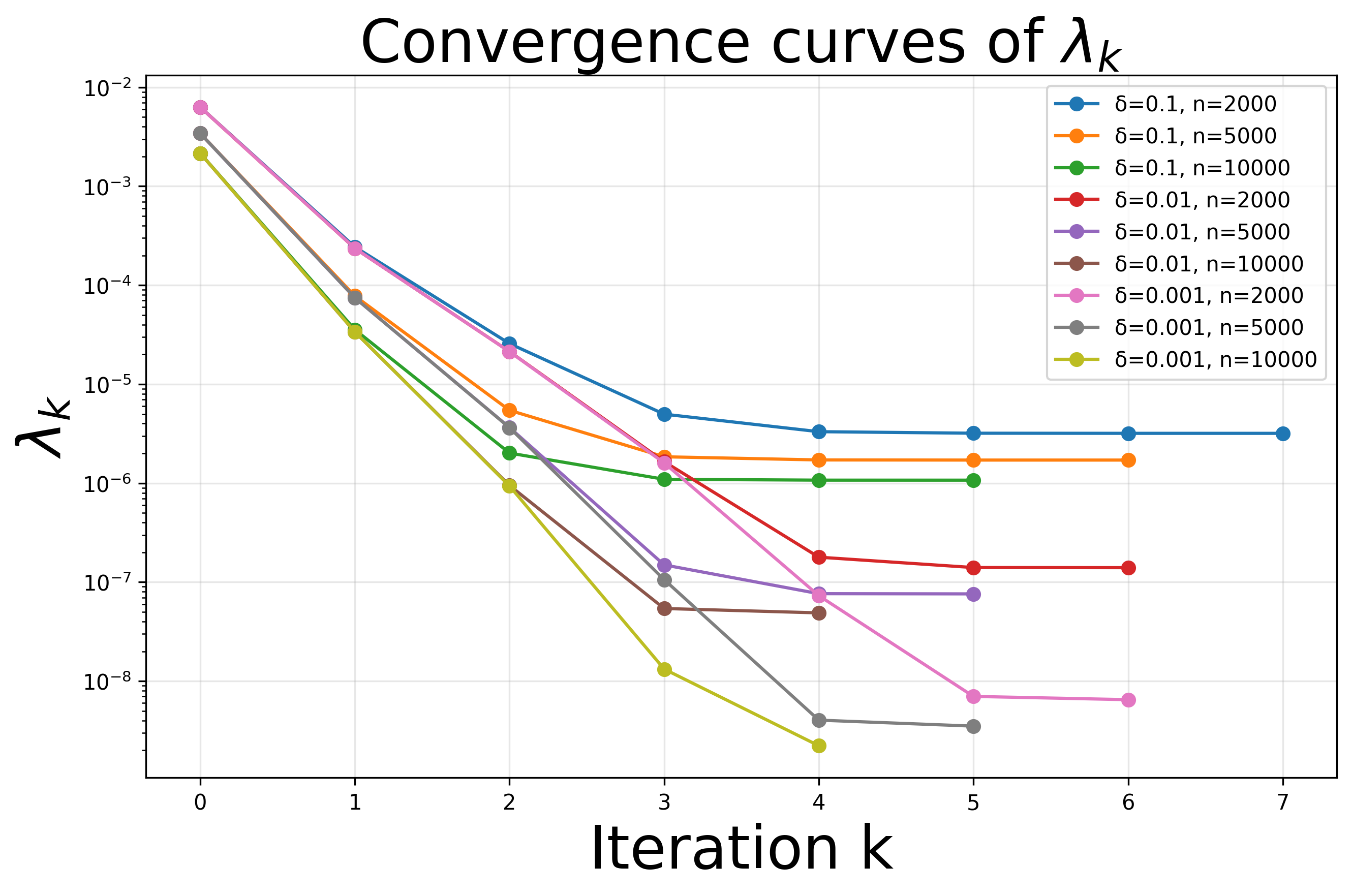} 
	\end{tabular}
	\caption{Convergence histories of the iteratively selected regularization parameter $\lambda_k$ for different noise levels $\delta\in\{0.1, 0.01, 0.001\}$ and discretization sizes $n=2000, 5000, 10000$. In all cases, $\lambda_k$ decreases rapidly during the first few iterations and then stabilizes, indicating fast and stable convergence of the parameter-selection scheme.}
	\label{fig:adaptive method}
\end{figure}

\begin{figure}[hbt!]
	\centering
	\setlength{\tabcolsep}{0pt}
	\begin{tabular}{cc}
		\includegraphics[width=0.3\textwidth]{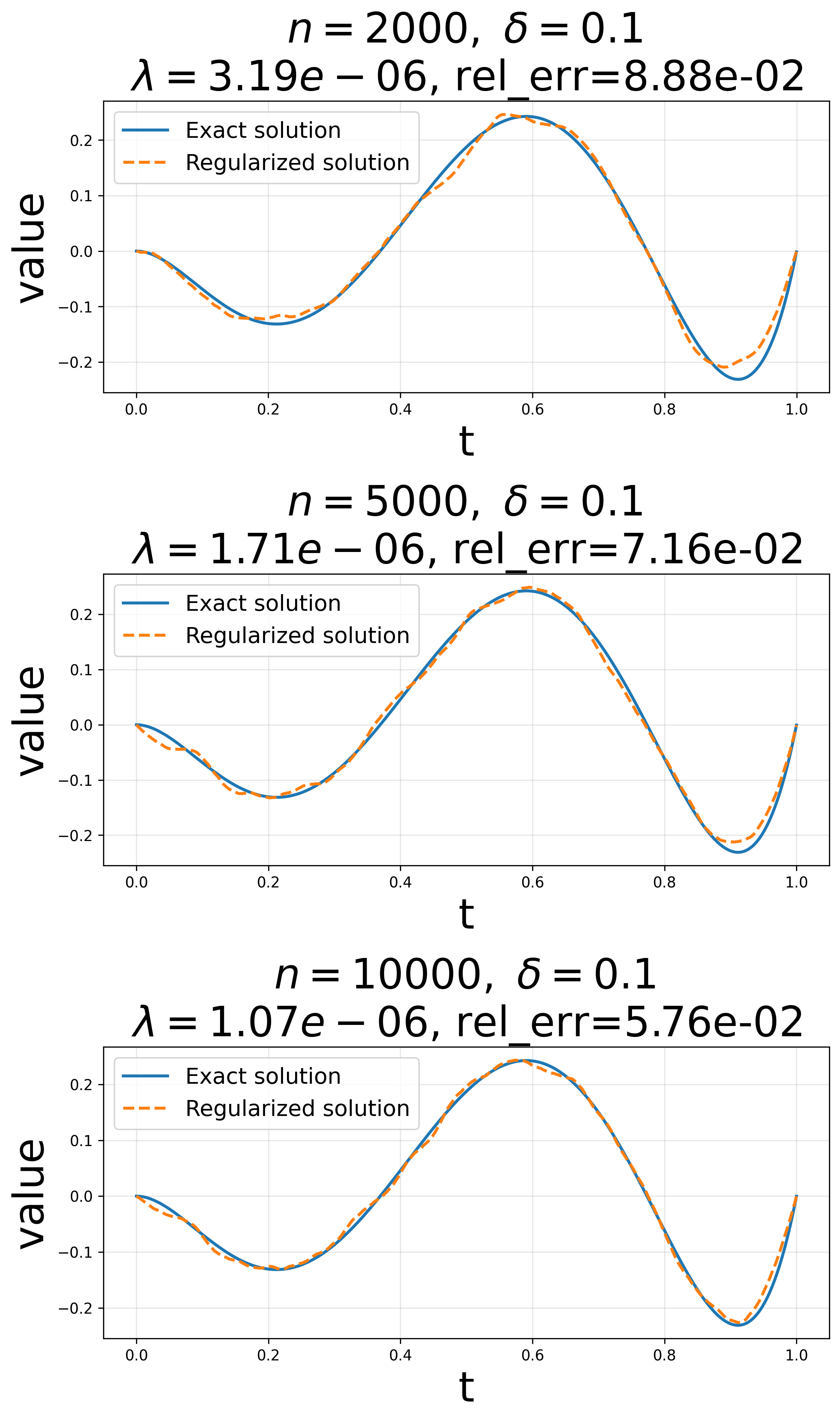} 
		\includegraphics[width=0.3\textwidth]{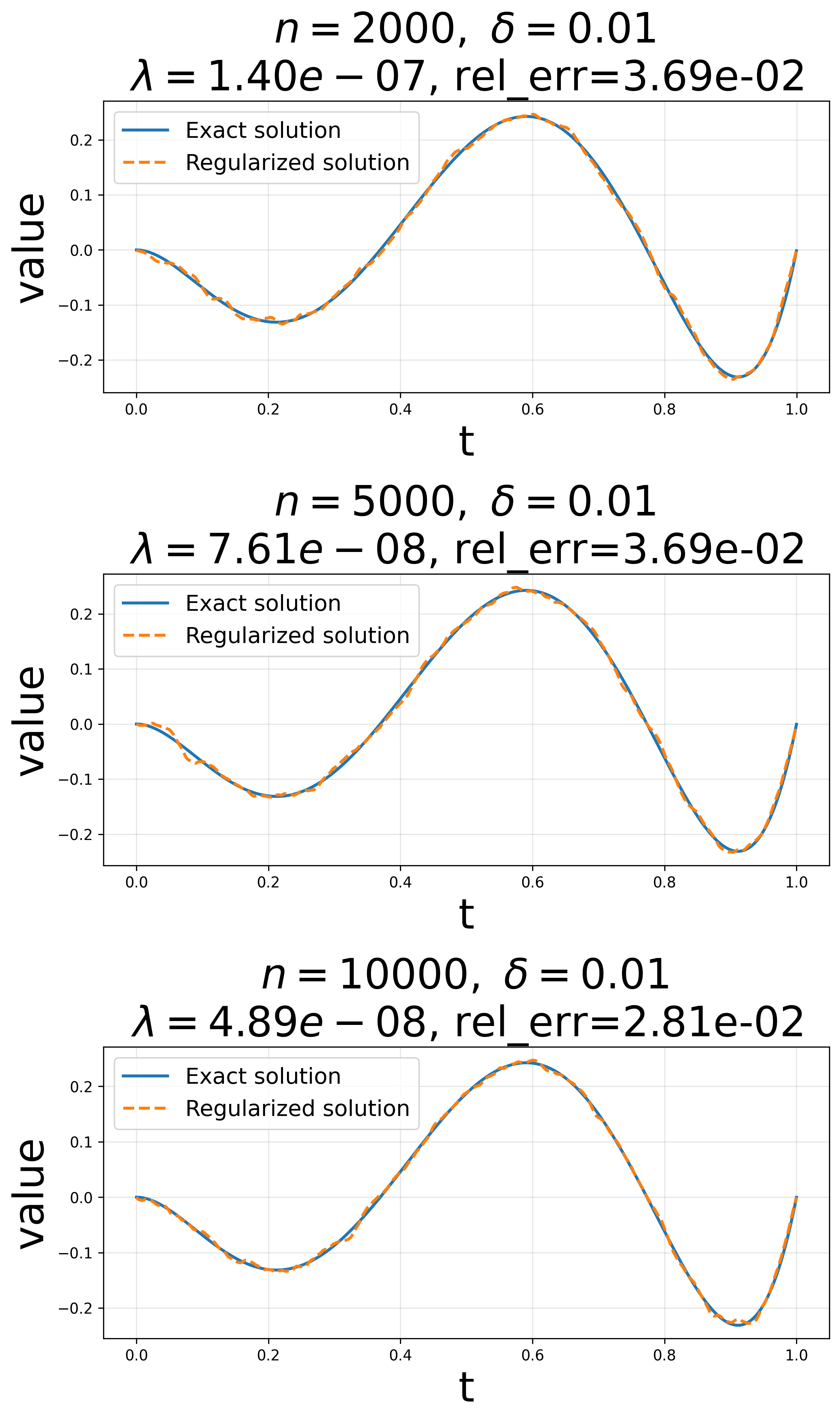} 
        \includegraphics[width=0.3\textwidth]{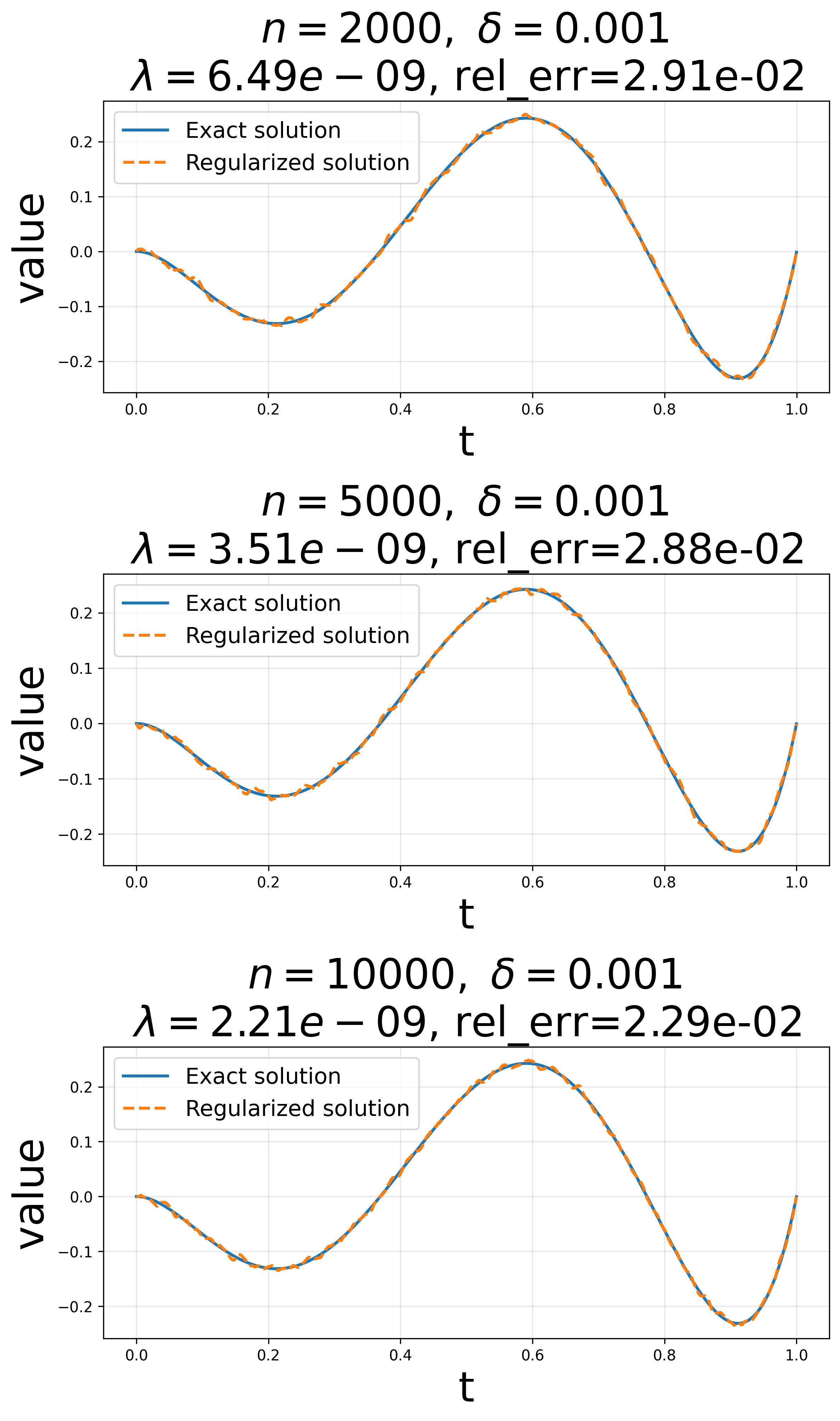} 
	\end{tabular}
	\caption{Reconstruction results at three noise levels, namely $\delta=0.001$, $\delta=0.01$, and $\delta=0.1$. For each noise level, the exact solution (solid blue) is compared with the regularized solution (dashed orange) for $n=2000$, $5000$, and $10000$. The reported values of $\lambda$ are the final regularization parameters obtained by the Algorithm \ref{algo}, and \texttt{rel\_err} denotes the relative reconstruction error $\|x_\lambda-x^\ast\|/\|x^\ast\|$. The plots show that the reconstructed solution remains close to the exact solution over all tested discretization levels \(n\), while the reconstruction accuracy improves as the noise level \(\delta\) decreases and as $n$ increases.}
	\label{fig:recon_sol}
\end{figure}

\begin{example}
This example is based on the IR Tools package \cite{Gazzola2019IRTools}. We consider the blurring of the Hubble space telescope image included in the package at two resolutions, namely \(100\times 100\) and \(180\times 180\) pixels. These correspond to problem sizes \(n=10000\) and \(n=32400\), respectively. Zero boundary conditions are imposed, which are natural here because the image is set against a black background.
\end{example}

\begin{itemize}
    \item Numerical validation of the spectral decay assumption and estimation of \(\alpha\).
\end{itemize}

To examine the spectral decay assumption,  we consider the speckle deblurring problem generated by \texttt{PRblurspeckle} at the two image resolutions introduced above. Since the blurring matrix \(A\)  generated by IR Tools is  a \texttt{psfMatrix} object, we convert it into an explicit matrix \(A\in\mathbb{R}^{n\times n}\), take \(W=I\), and compute the positive generalized eigenvalues \(\rho_k\) of \(A^{\top}A\psi=\rho W\psi\). Ordering \(\rho_k\) decreasingly, we estimate the decay exponent from the log-log fit \(\log \rho_k\approx \log C-\alpha\log k\) over the intermediate range \(k=6,\dots,\min(400,\lfloor m/2\rfloor)\), where \(m\) is the number of retained positive eigenvalues. The resulting \(\hat\alpha\), together with the envelope \(C_{\mathrm{upper}}k^{-\hat\alpha}\), where \(C_{\mathrm{upper}}=\max_{1\le k\le m}\rho_k k^{\hat\alpha}\), is then used to assess the plausibility of the algebraic decay condition \(\rho_k \leq Ck^{-\alpha}\). Figure \ref{ex2fig:alpha} shows the decay of the generalized eigenvalues \({\rho_k}\) for two discretization levels, \(n=10000\) and \(n=32400\), with corresponding fitted exponents \(\hat{\alpha}=1.0933\) and \(\hat{\alpha}=1.5336\), respectively.

\begin{figure}[hbt!]
    \centering
    \begin{subfigure}{0.4\textwidth}
        \centering
        \includegraphics[width=\linewidth]{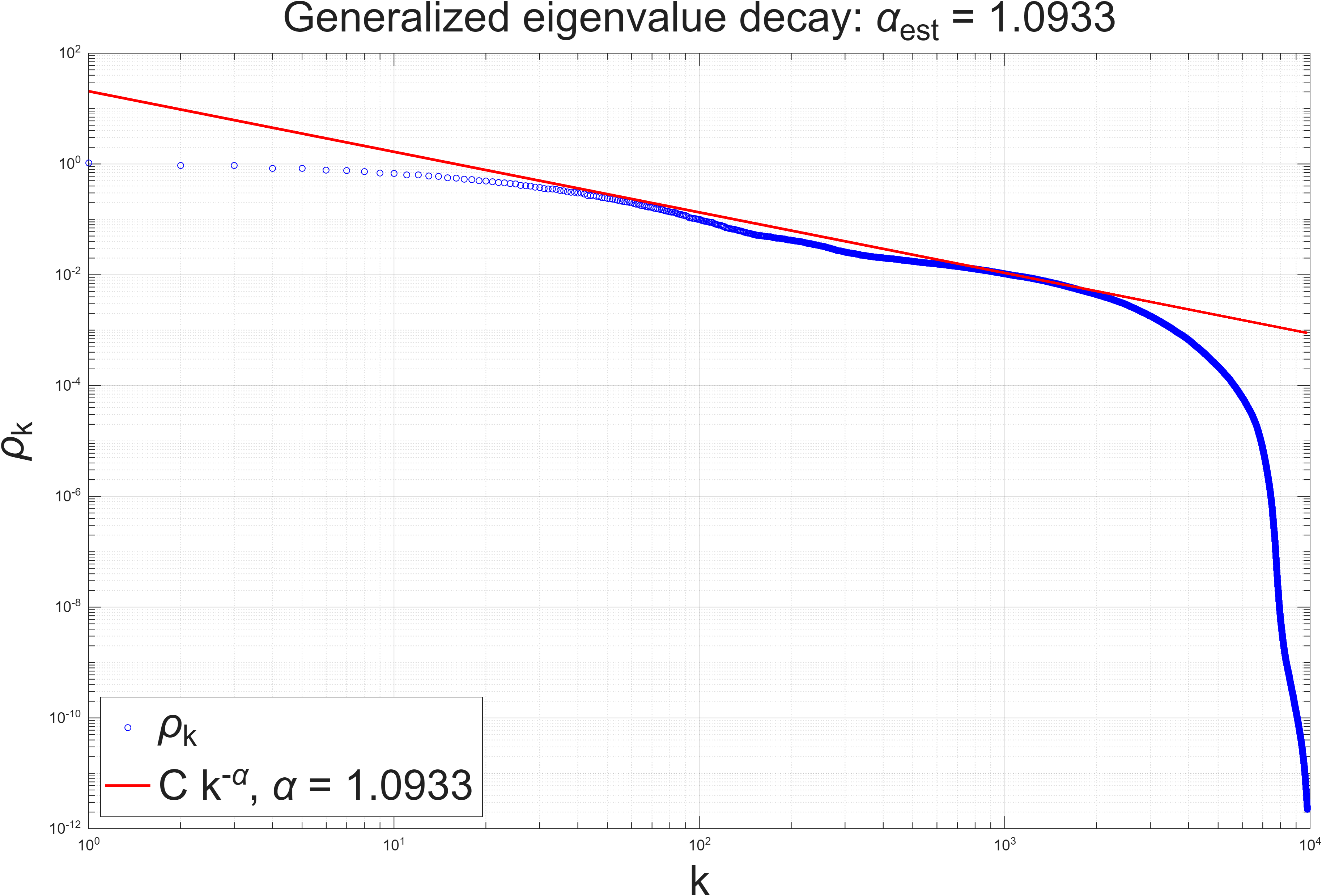}
    \end{subfigure}
    \begin{subfigure}{0.4\textwidth}
        \centering
        \includegraphics[width=\linewidth]{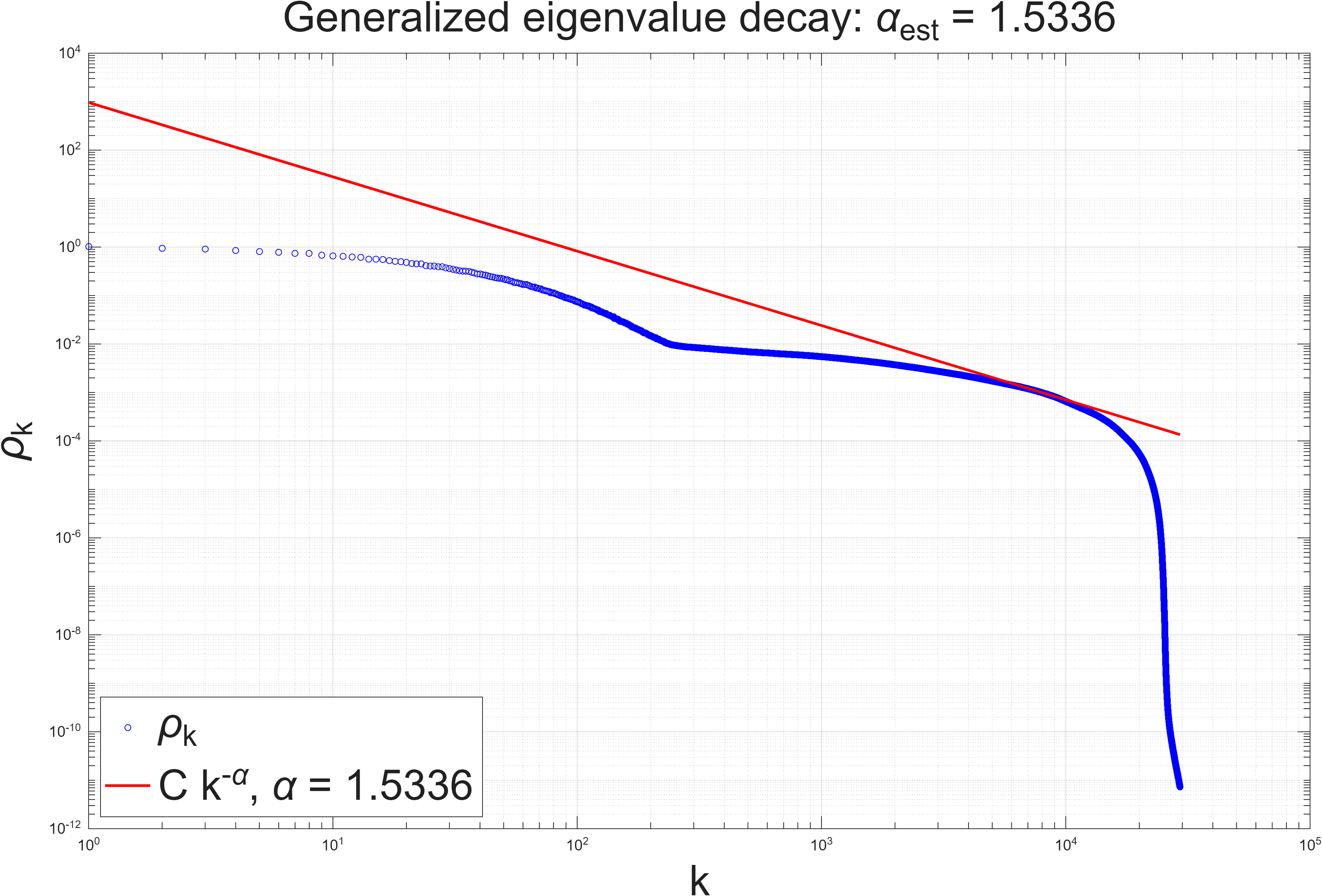}
    \end{subfigure}
    \caption{Decay of the generalized eigenvalues \({\rho_k}\) versus the index \(k\) on a double-logarithmic scale for two discretization levels: \(n=10000\) (left) and \(n=32400\) (right).} 
    \label{ex2fig:alpha}
\end{figure}

\begin{itemize}
    \item Near-optimality of the a priori parameter choice (\ref{opt-parameter}) in terms of the output error.
\end{itemize}

As in Example \ref{nr:int}, the experiment is carried out by computing the regularized solutions over a prescribed parameter grid and comparing the resulting output errors with the theoretically predicted parameter. The figure \ref{ex2_fig:parameter}  shows the same overall behavior of the output error \(\frac{1}{\sqrt{n}}\|Ax_\lambda-Ax^\ast\|\) with respect to the regularization parameter \(\lambda\). In both cases, the error decreases as \(\lambda\) increases from very small values, attains its minimum in an intermediate range, and then rises rapidly for larger \(\lambda\). The parameter predicted by the a priori choice rule \eqref{opt-parameter}, with the implicit constant in the \(O(\cdot)\) term fixed at \(6\), is indicated by the red star and is seen to lie close to the minimum region of the error curve.

\begin{figure}[hbt!]
    \centering
    \begin{subfigure}{0.4\textwidth}
        \centering
        \includegraphics[width=\linewidth]{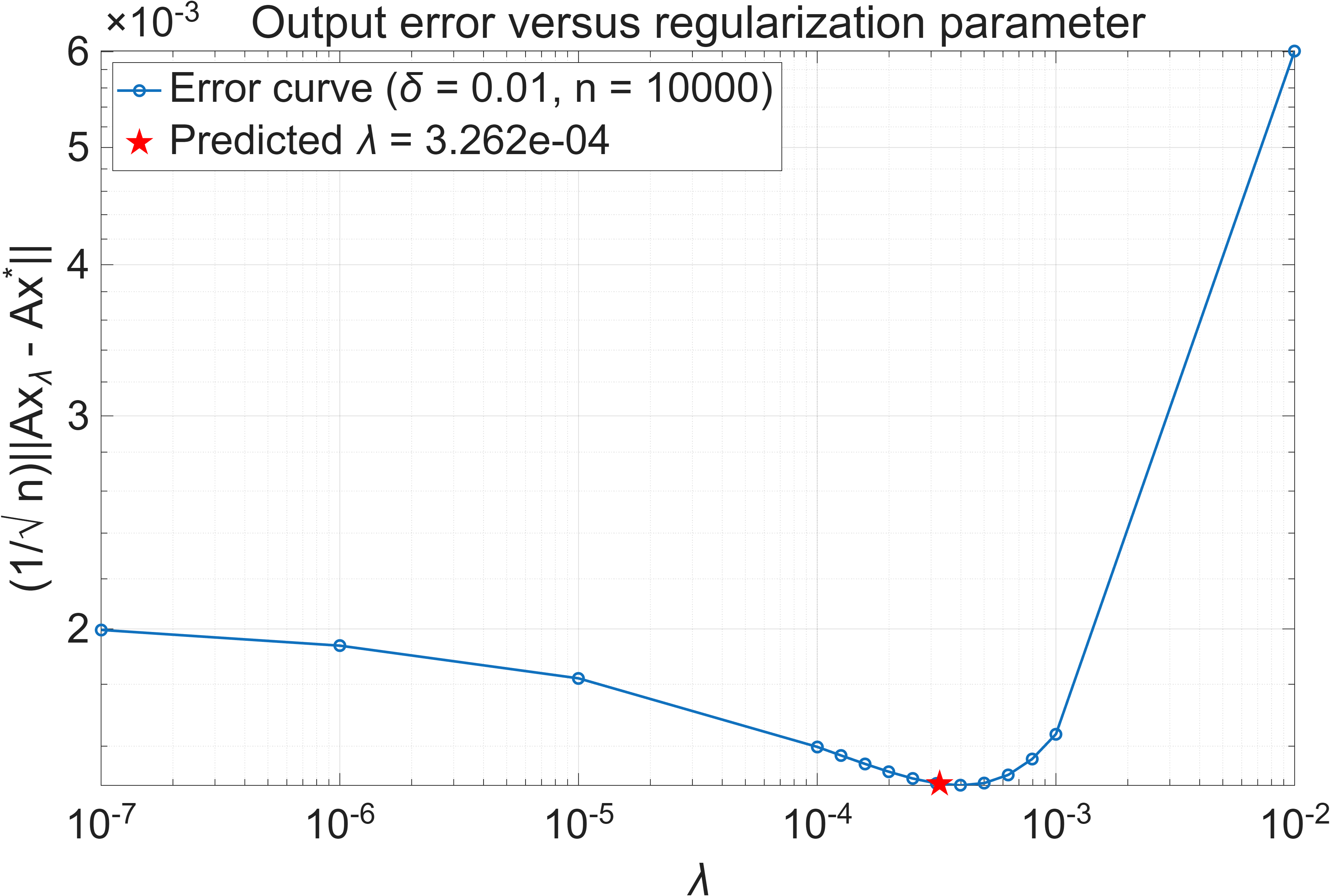}
    \end{subfigure}
    \begin{subfigure}{0.4\textwidth}
        \centering
        \includegraphics[width=\linewidth]{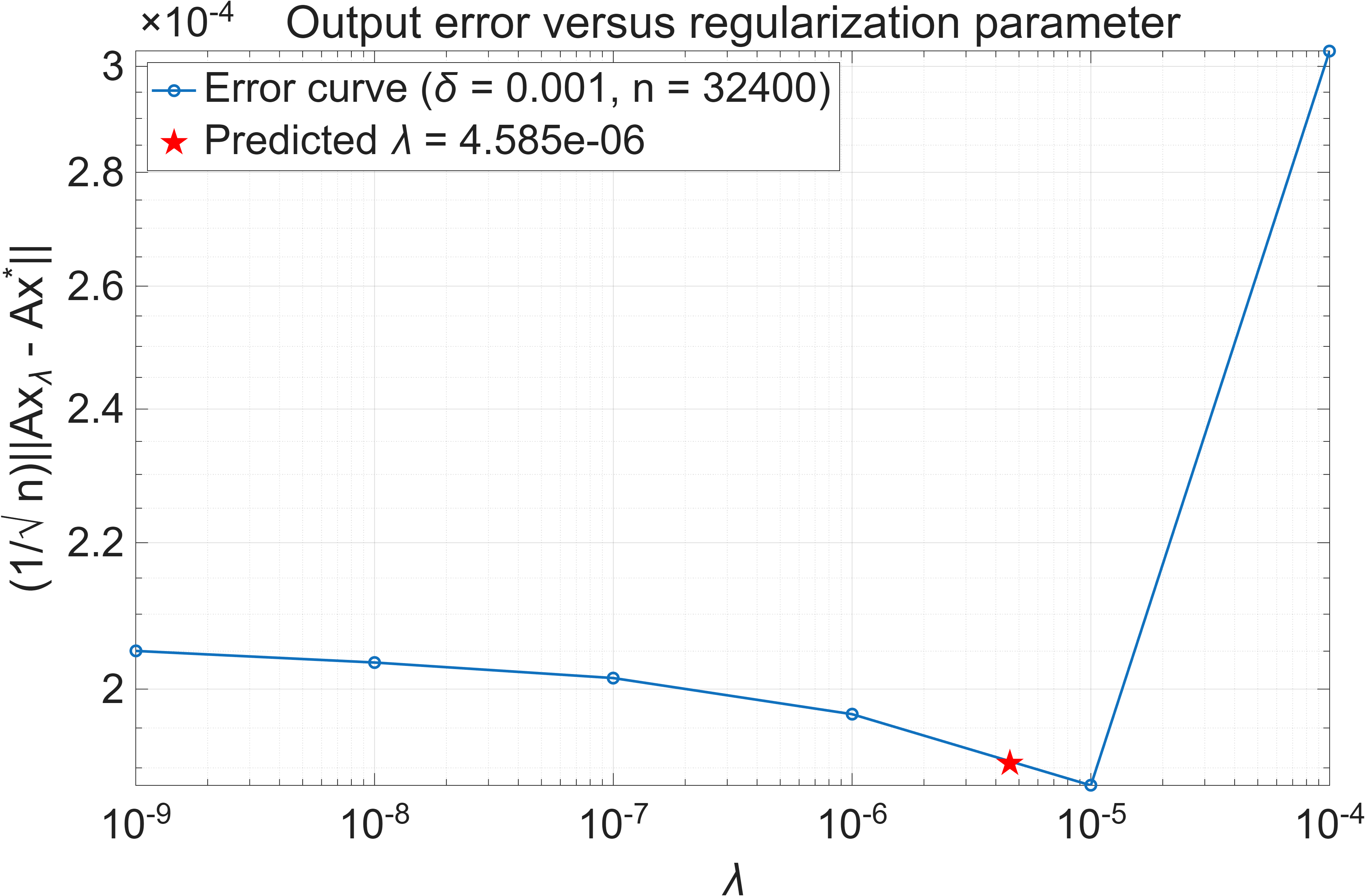}
    \end{subfigure}
    \caption{Comparison of output errors \(\frac{1}{\sqrt n}\|Ax_\lambda-Ax^\ast\|\)  versus the regularization parameter \(\lambda\) for two representative settings: \(\delta=0.01\), \(n=10000\), and \(\delta=0.001\), \(n=32400\).  In both cases, the error first decreases, then attains a minimum in an intermediate parameter range, and finally increases rapidly for larger \(\lambda\). The predicted parameters \(\lambda_{\mathrm{pred}}\), marked by red stars, lie close to the minima of the corresponding error curves, indicating that the a priori parameter-choice rule \eqref{opt-parameter} is nearly optimal.} 
    \label{ex2_fig:parameter}
\end{figure}

\begin{itemize}
    \item Empirical distribution of the output error
\end{itemize}

To examine the statistical behavior of the output error, we conduct a Monte Carlo experiment with \(\delta=0.01\) and \(n=10000\), using the parameter rule \eqref{opt-parameter} to determine a nearly optimal \(\lambda\). We generate \(10000\) independent noise realizations and record the empirical errors \(\frac{1}{\sqrt n}\|Ax_\lambda-Ax^\ast\|.\) The histogram in Figure \ref{ex2_fig:histogram_qq} shows a strong concentration of the errors, while the Q-Q plot is close to linear except for slight tail deviations.  This indicates that the empirical distribution of the output error is well approximated by a Gaussian law and exhibits an approximately exponential, more precisely sub-Gaussian, tail decay. These observations provide qualitative support for the concentration estimate \eqref{eq:Pb-A}. 

\begin{figure}[hbt!]
    \centering
    \begin{subfigure}{0.4\textwidth}
        \centering
        \includegraphics[width=\linewidth]{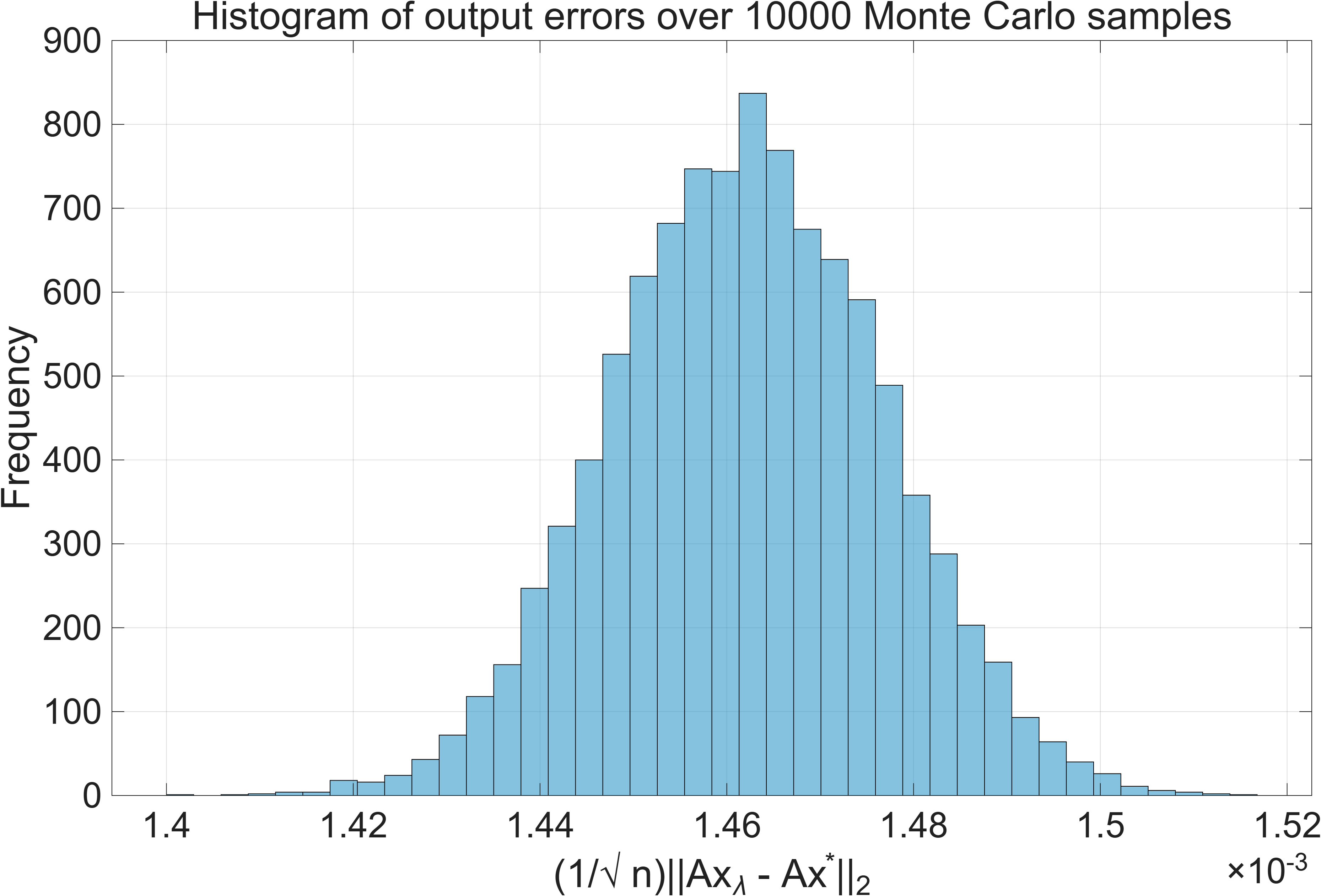}
        
    \end{subfigure}
    \begin{subfigure}{0.4\textwidth}
        \centering
        \includegraphics[width=\linewidth]{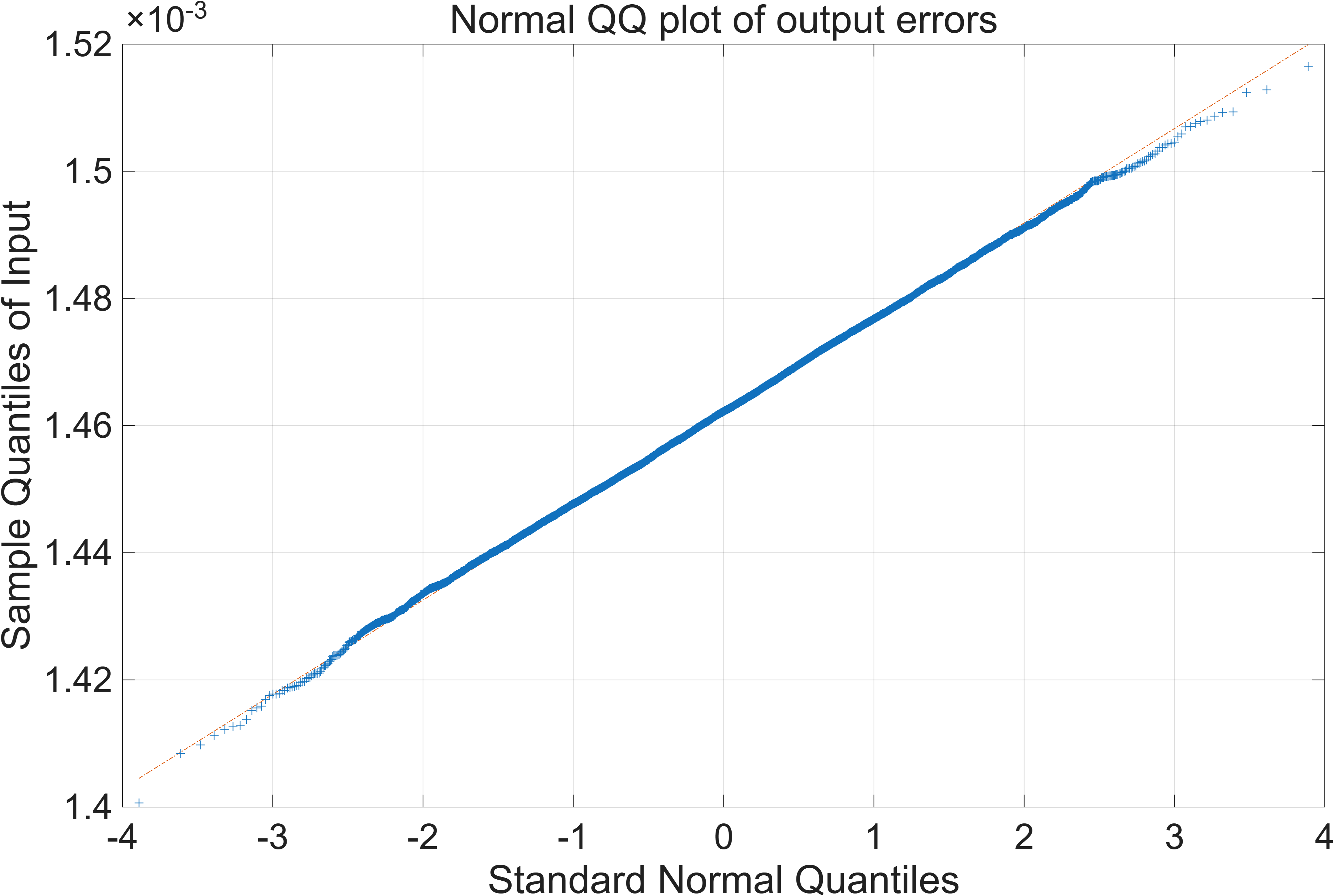}
    \end{subfigure}
    \caption{Empirical distribution of \(\frac{1}{\sqrt n}\|Ax_\lambda-Ax^\ast \|\) over \(10000\) Monte Carlo samples: histogram (left) and normal QQ plot (right). The results indicate an approximately Gaussian distribution with mild tail deviations.} 
    \label{ex2_fig:histogram_qq}
\end{figure}

\begin{itemize}
    \item Performance of the adaptive parameter-choice rule
\end{itemize}

As in Example \ref{nr:int}, we apply the adaptive parameter-choice Algorithm \ref{algo} to the cases \(\delta=0.01\), \(n=10000\), and \(\delta=0.001\), \(n=32400\); see Figures \eqref{fig:a}--\eqref{fig:f}. In the implementation, the iteration is terminated when the relative change in the parameter satisfies
\[
\frac{|\lambda_k-\lambda_{k-1}|}{\lambda_k}\leq \mathrm{tol}_\lambda,
\qquad \mathrm{tol}_\lambda=10^{-3}.
\]
For \(\delta=0.01\) and \(n=10000\), the parameter sequence \(\{\lambda_k\}\) decreases quickly and stabilizes at about \(2.23\times 10^{-4}\), while the output error \((1/\sqrt n)\|Ax_{\lambda_k}-Ax^\ast\|\) decreases rapidly in the first few iterations and then levels off. For the smaller noise level \(\delta=0.001\) with \(n=32400\), the iteration selects a substantially smaller parameter, approximately \(1.33\times 10^{-6}\), and achieves a lower final output error. Figures \eqref{fig:e}–\eqref{fig:f} show that the adaptive method stably recovers the principal image features in both cases, while the lower noise level yields a smaller selected regularization parameter and a visibly sharper reconstruction. These results support the stability and effectiveness of the proposed adaptive procedure.

\begin{figure}[hbt!]
    \centering
    \begin{subfigure}{0.4\textwidth}
        \centering        \includegraphics[width=\linewidth]{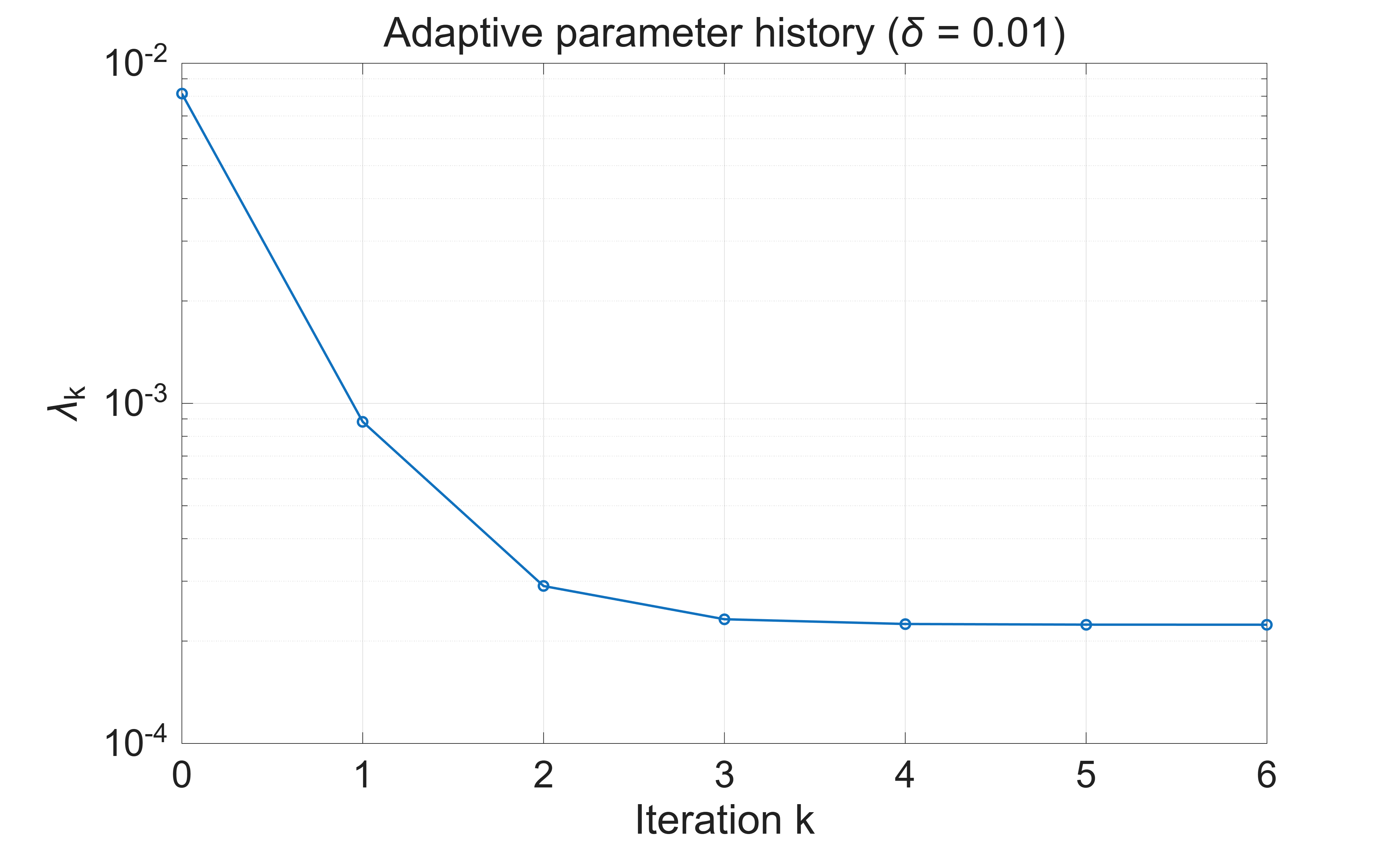}
        \caption{Convergence history of the adaptive regularization parameter \(\lambda_k\) for \(\delta=0.01\) and \( n = 10000\).}
        \label{fig:a}
    \end{subfigure}
    \begin{subfigure}{0.4\textwidth}
        \centering
    \includegraphics[width=\linewidth]{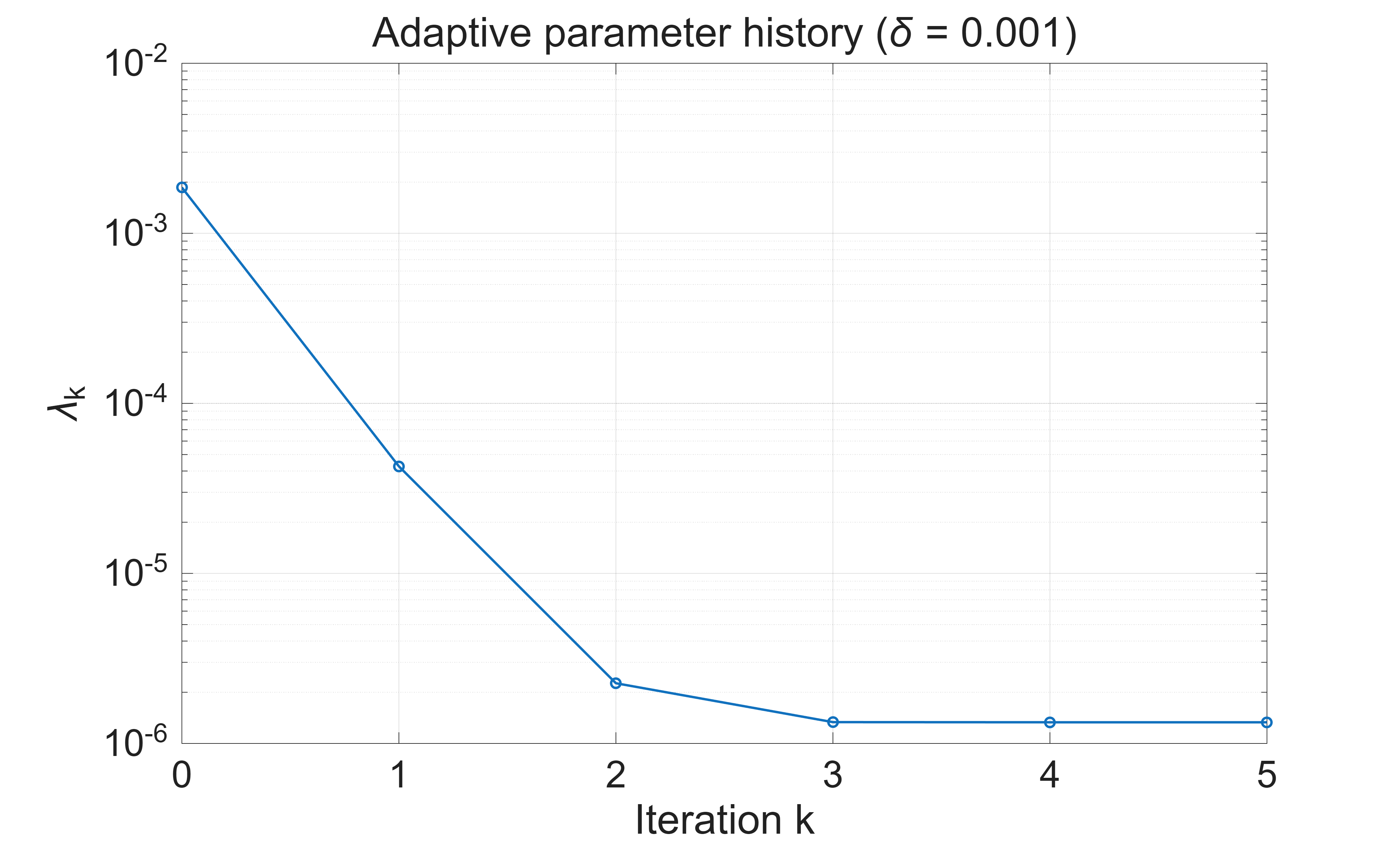}
        \caption{Convergence history of the adaptive regularization parameter \(\lambda_k\) for \(\delta=0.001\) and \(n = 32400\).}
    \end{subfigure}

    \begin{subfigure}{0.4\textwidth}
        \centering
        \includegraphics[width=\linewidth]{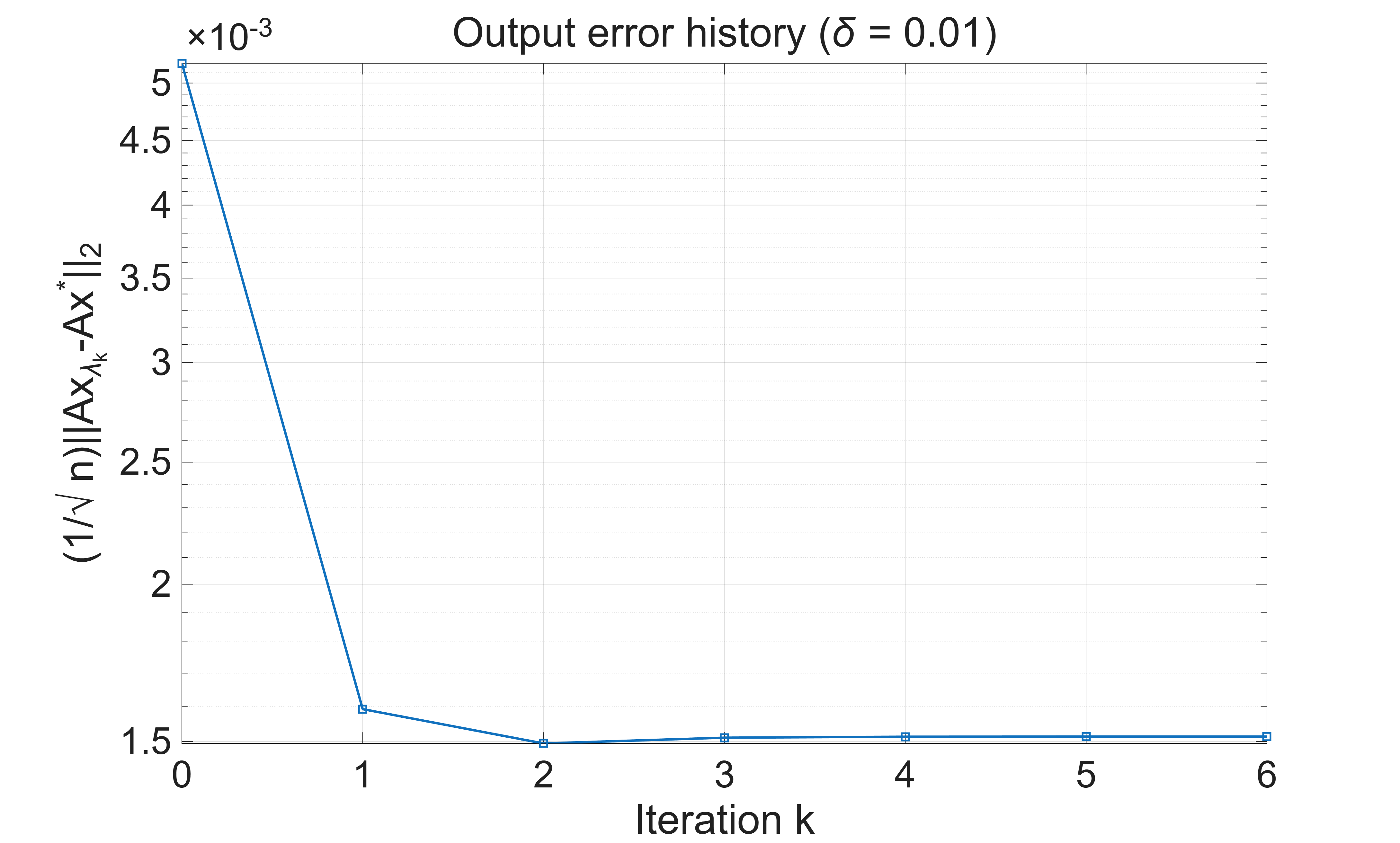}
        \caption{The output error \((1/\sqrt n)\|Ax_{\lambda_k}-Ax^\ast\|\) versus iteration number for \(\delta=0.01\) and \( n = 10000\).}
    \end{subfigure}
    \begin{subfigure}{0.4\textwidth}
        \centering
        \includegraphics[width=\linewidth]{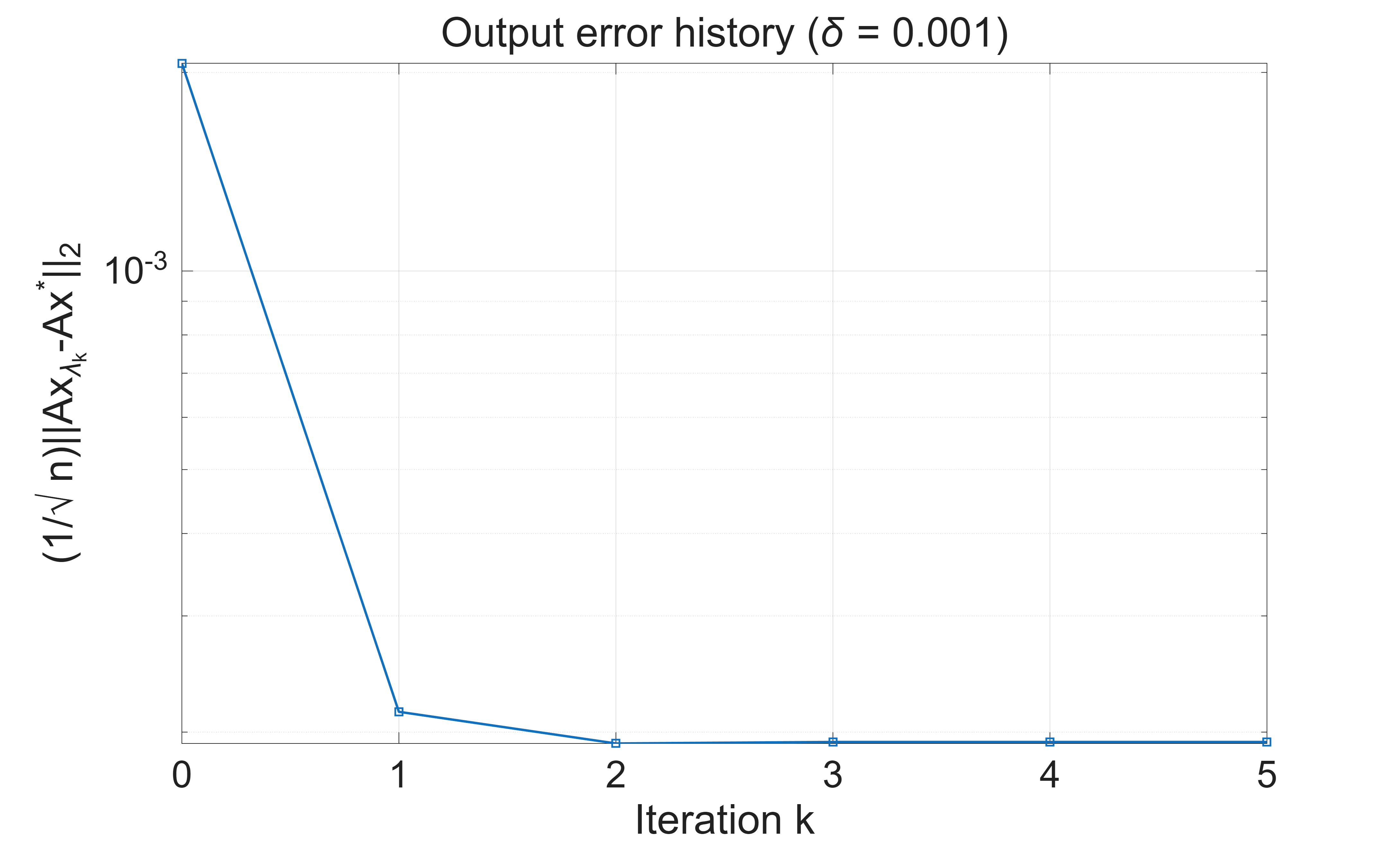}
        \caption{ The output error \((1/\sqrt n)\|Ax_{\lambda_k}-Ax^\ast\|\) versus iteration number for  \(\delta=0.001\) and \(n = 32400\).}
    \end{subfigure}

    \begin{subfigure}{0.48\textwidth}
        \centering
        \includegraphics[width=\linewidth]{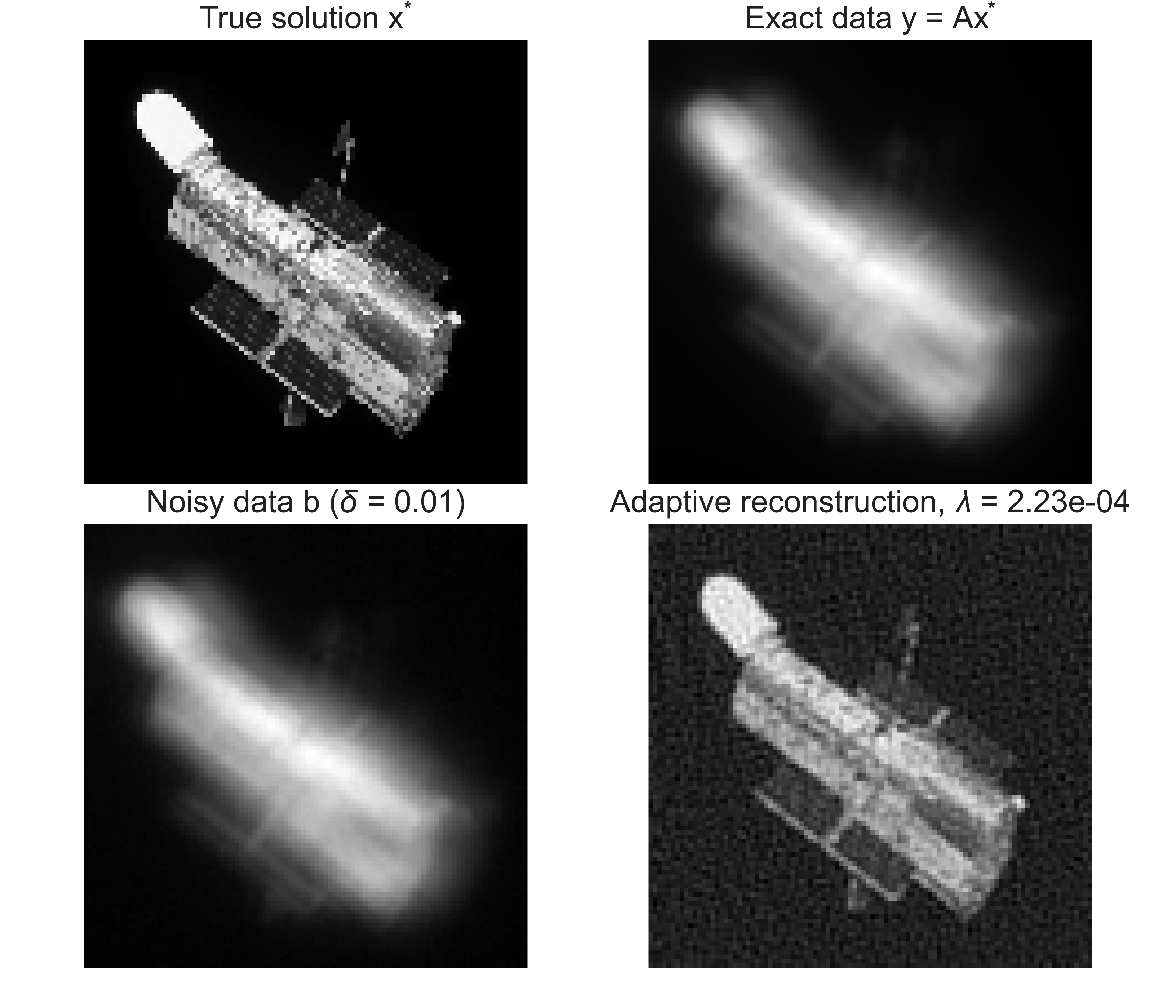}
        \caption{Comparison of the true solution, exact data, noisy data, and adaptive reconstruction for \(\delta=0.01\) and \( n = 10000\)}
        \label{fig:e}
    \end{subfigure}
    \hfill
    \begin{subfigure}{0.48\textwidth}
        \centering
        \includegraphics[width=\linewidth]{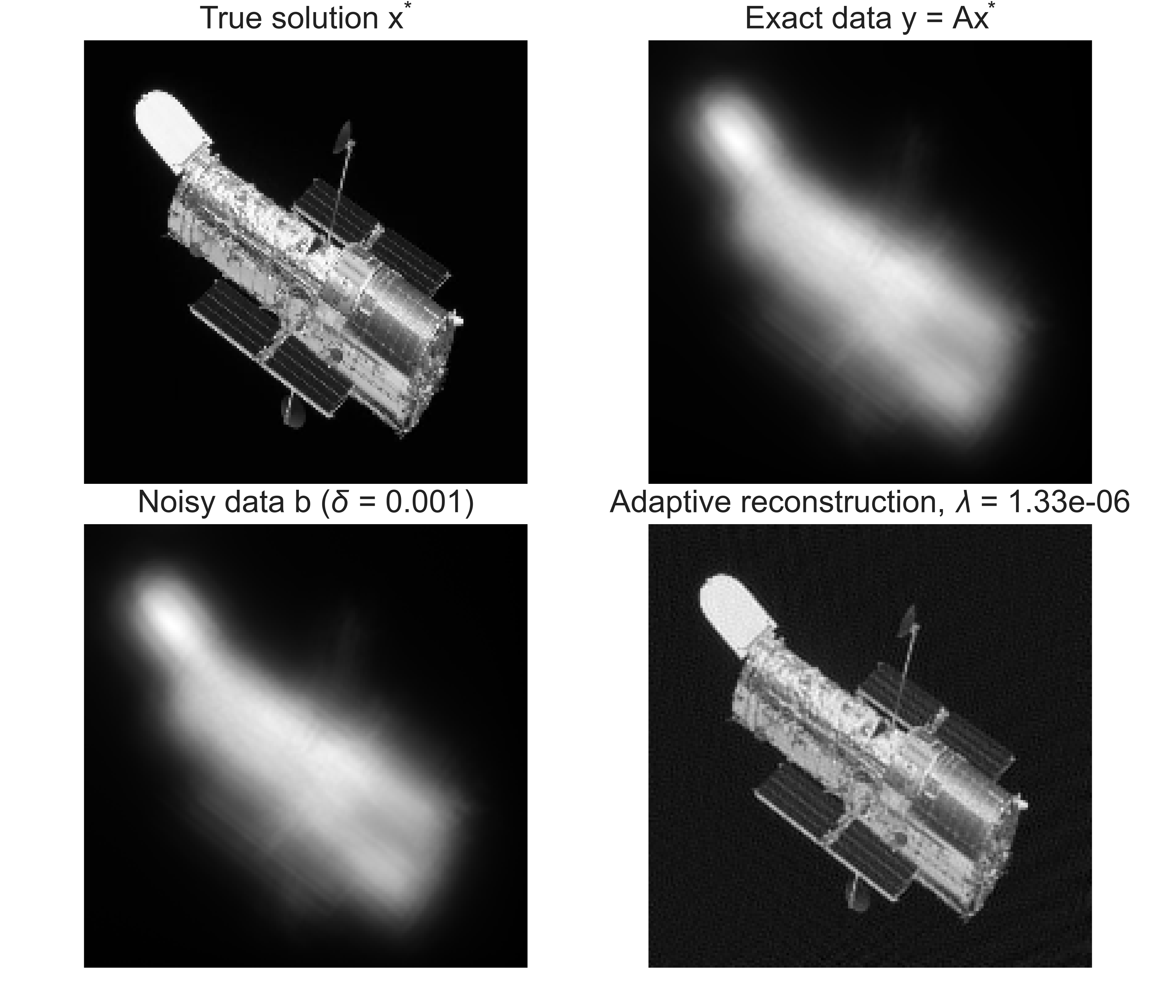}
        \caption{Comparison of the true solution, exact data, noisy data, and adaptive reconstruction for \(\delta=0.001\) and \(n = 32400\).}
        \label{fig:f}
    \end{subfigure}
     \label{fig:ex2_ad}
\end{figure}


\end{document}